\documentclass[leqno,11pt,a4paper]{article}
\usepackage[left=2.0cm, top=2.2cm,bottom=2.5cm,right=2.0cm]{geometry}
\usepackage{amsmath,amssymb,amsthm,mathrsfs,calc,graphicx}
\usepackage{dsfont}
\usepackage[british]{babel}
\usepackage[utf8]{inputenc}

\numberwithin{equation}{section}

\newtheorem {theorem}{Theorem}[section]
\newtheorem {proposition}[theorem]{Proposition}
\newtheorem {lemma}[theorem]{Lemma}

{\theoremstyle{definition}
\newtheorem{definition}[theorem]{Definition}
}
{\theoremstyle{theorem}
\newtheorem {remark}[theorem]{Remark}
\newtheorem {example}[theorem]{Example}
}

\def\ba{\begin{array}}
\def\ea{\end{array}}
\def\bea{\begin{eqnarray} \label}
\def\eea{\end{eqnarray}}
\def\be{\begin{equation} \label}
\def\ee{\end{equation}}
\def\bit{\begin{itemize}}
\def\eit{\end{itemize}}
\def\ben{\begin{enumerate}}
\def\een{\end{enumerate}}

\def\EE{\mathbb{E}}

\def\NN{\mathbb{N}}
\def\PP{\mathbb{P}}

\def\RR{\mathbb{R}}

\def\VV{\mathbb{V}}

\def\sumnod{\sum_{  i_1,\ldots,i_k=1 }^{m}\!\!\!\!\!\!^{\neq}\quad\,}

\def\g{\gamma}
\def\d{\delta}

\def\s{\sigma}

\def\cC{\mathcal{C}}
\def\cF{\mathcal{F}}

\def\cH{\mathcal{H}}

\def\cL{\mathcal{L}}

\def\cZ{\mathscr{Z}}

\def\sZ{\mathcal{Z}}

\def\dint{\textup{d}}
\def\s{\star}
\def\ts{\,\widetilde{\star}}
\def\cGam{\overline{\Gamma}}
\def\dist{\textup{dist}}

\begin{document}

\title{\bfseries Gamma limits and $U$-statistics on the Poisson space}

\author{Giovanni Peccati\footnotemark[1]\ \ and Christoph Th\"ale\footnotemark[2]\ \footnotemark[3]}

\date{}
\renewcommand{\thefootnote}{\fnsymbol{footnote}}
\footnotetext[1]{Luxembourg University, Mathematics Research Unit, Campus Kirchberg, G 221, L-1359 Luxembourg. E-mail: giovanni.peccati@gmail.com}

\footnotetext[2]{Ruhr-University Bochum, Faculty of Mathematics, NA 3/68, D-44781 Bochum, Germany. E-mail: christoph.thaele@rub.de}

\footnotetext[3]{The second author has been supported by the German Research Foundation (DFG) via SFB/TR-12 ``Symmetries and Universality in Mesoscopic Systems''.}

\maketitle

\begin{abstract}
Using Stein's method and the Malliavin calculus of variations, we derive explicit estimates for the Gamma approximation of functionals of a Poisson measure. In particular, conditions are presented under which the distribution of a sequence of multiple Wiener-It\^o stochastic integrals with respect to a compensated Poisson measure converges to a Gamma distribution. As an illustration, we present a quantitative version and a non-central extension of a classical theorem by de Jong in the case of degenerate $U$-statistics of order two. Several multidimensional extensions, in particular allowing for mixed or hybrid limit theorems, are also provided.
\bigskip
\\
{\bf Keywords}. {Chaos; Contraction; Gamma distribution; De Jong's theorem; Malliavin calculus; Mixed limit theorem; Multiple stochastic integral; Non-central limit theorem; Poisson process; Stein's method; U-statistic.}\\
{\bf MSC2010}. Primary  60F05, 60G55; Secondary 60H05, 60H07, 62E20.
\end{abstract}

\section{Introduction}

The use of the Malliavin calculus of variations in order to deduce limit theorems for non-linear functionals of random measures has recently become a relevant direction of research, one reason for that being the many successful applications in geometric probability or stochastic geometry. Apart from a few exceptions, most contributions to this topic fall into the two categories of normal and Poisson approximations; see \cite{DFRV,LRP,LRP2,LPST,PSTU2010,PecZheng,ReitznerSchulte,Viq} for distinguished examples of the former class, mostly based on the use of the Stein's method (cf.\ \cite{NPBook}); see \cite{BourPec,PecPoisson,ST12} for references based on the combination of Malliavin calculus and of the Chen-Stein method for Poisson approximations. We also refer to \cite{EV} for recent extensions to general absolutely continuous distributions having support equal to the real line.

\smallskip

The aim of the present paper is to provide the first array of results concerning limit theorems on the Poisson space, where the limit distribution is absolutely continuous and has support contained in a proper subset of $\RR$. More precisely, we are interested in probabilistic approximations where the limiting random variable has a centred Gamma distribution $\cGam_\nu$ with parameter $\nu>0$. We say that a random variable $G(\nu)$ has distribution $\cGam_\nu$ if $G(\nu)\overset{d}{=}2F(\nu/2)-\nu$, where $F(\nu/2)$ has a usual Gamma distribution with mean and variance both equal to $\nu/2$ (here and throughout $\overset{d}{=}$ stands for equality in distribution). If $\nu\geq 1$ is an integer, then $\cGam_\nu$ reduces to the centred $\chi^2$-distribution with $\nu$ degrees of freedom. We remark that the support of $\cGam_\nu$ is given by the half-line $[-\nu, +\infty)$, and that the first four moments of $\cGam_\nu$ are $0$, $2\nu$, $8\nu$ and $12\nu^2+48\nu$, respectively. 
 We will often meet these expressions 
in the discussion to follow. 

\smallskip

Our main contribution is the general estimate stated in Theorem \ref{thm:SteinBound}, which involves Malliavin operators and is obtained by means of Stein's method, allowing one to measure the distance between the law of a given Poisson functional and $\cGam_\nu$. This estimate is applied to deduce explicit sufficient conditions for Gamma limit theorems involving sequences of multiple Wiener-It\^o stochastic integrals. Our analysis is significantly inspired by \cite{NourdinPeccatiNoncentral,NourdinPeccatiWienerChaos}, where the problem addressed in the present paper was first studied in the framework of non-linear functionals of general Gaussian fields. However, due to the combinatorial complications one has to face when dealing with point measures, our paper contains a number of new subtle computations related to the explicit estimation of Malliavin operators on configuration spaces. One specific problem we will have to deal with is that the solution of the Stein's equation associated with the law of $G(\nu)$ is not differentiable at $x= -\nu$. Thus, in order to obtain bounds that are well-suited for our applications (which may involve random variables possibly taking values in $(-\infty, -\nu)$), we will have to combine techniques recently introduced by Schulte \cite{SchulteKolm} with classical {isometric formulae} borrowed from the standard reference \cite{privaultbook}; see Proposition \ref{l:newbound} below. One should note that, in view of the exact chain rules that are available on a Gaussian space, the non-differentiability of the Stein solution in one point is immaterial when studying the Gamma approximation of smooth functionals of a Gaussian field; see again \cite{NourdinPeccatiNoncentral,NourdinPeccatiWienerChaos}.

\smallskip

As an illustration, we will include some applications to non-central limit theorems for sequences of degenerate (in the sense of Hoeffding) $U$-statistics. Our findings generalize several classic result in the field; cf.\ \cite{BhGh,JJ}. In particular, we derive a quantitative and a non-central version of a famous theorem by P.\ de Jong \cite{DJ,DJMulti}. Our analysis also contains a quantitative version of a non-central result recently discussed by Reitzner and Schulte \cite[Section 5.1]{ReitznerSchulte}.

\smallskip

Finally, to demonstrate the flexibility and scope of our approach, we will show that our analysis can naturally be extended to a multidimensional framework. We will not only obtain multidimensional Gamma limit theorems, but also {\it mixed} or {\it hybrid} results, where the multidimensional limit distribution is composed both of Gamma and of normal or Poisson components. This kind of limit theorems heavily relies on our use of Malliavin operators. We are not aware of any other available technique allowing one to deduce general mixed limit results, such as the ones deduced in the present paper. We shall see that our findings are a refinement of the `Portmanteau inequalities' recently obtained by Bourguin and Peccati in \cite{BourPec}. In this respect, we stress that our results will implicitly yield a collection of sufficient conditions in order to have that two sequences of Poisson functionals are asymptotically independent. This provides a new contribution to the difficult and mostly open problem of characterizing the asymptotic and non-asymptotic independence of functionals of a Poisson measure; see e.g.\ \cite{Privault, RosSam}.

\medskip

The remainder of the paper is organized as follows. In Section \ref{sec:results} we present our results in full generality. Some background material is collected in Section \ref{sec:background}, whereas the final Section \ref{sec:proofs} contains detailed proofs, as well as some ancillary technical results.

\section{Presentation of the results}\label{sec:results}

We will now present an overview of the main findings of the paper. To enhance the readability of our text, we have gathered together in Section \ref{sec:background} definitions, notation and relevant results from the literature.

\subsection{General limit theorems}\label{ss:gen}

Every random object considered below is defined on a suitable probability space $(\Omega, \mathcal{F}, \PP)$. The approximation results obtained in the present paper deal with (real-valued) functionals of a Poisson measure $\eta$ on some Polish space $(\sZ,\cZ)$ having non-atomic and $\sigma$-finite control $\mu$; see Section \ref{sec:background}-(I). We will assume that these functionals are square-integrable random variables. To measure the distance between the distribution of a functional $F$ of $\eta$ and that of a centred Gamma random variable $G(\nu)$, we shall use the (pseudo-) metric $d_3$, which is defined as follows: for every pair of square-integrable random variables $X,Y$, we put $$d_3(X,Y)=\sup_{h\in\cH^3}\big|\EE[h(X)]-\EE[h(Y)]\big|,$$ where $\cH^3:=\{h\in\cC^3:\|h^{(j)}\|_\infty\leq 1,\,j\in\{1,2,3\}\}$ (with $h^{(j)}$ the derivative of order $j$ of $h$), and where $\cC^3$ is the space of thrice differentiable functions on $\RR$ having bounded derivatives. We notice that the topology induced by $d_3$ is stronger than the topology induced by convergence in distribution, which implies that if $d_3\big(F_n,G(\nu)\big)\to 0$, as $n\to\infty$, for some sequence of functionals $F_n$, then the distribution of $F_n$ converges to $\cGam_\nu$. By a slight abuse of notation, and to stress the role of the underlying Gamma distribution, we shall often write $d_3(F,\cGam_\nu)$ instead of $d_3\big(F,G(\nu)\big)$.

\smallskip

For $q\geq 1$, we write $L^2(\mu^q)$ to indicate the Hilbert space of Borel-measurable functionals on $\sZ^q$ that are square-integrable with respect to $\mu^q$. We also use the following special notation: $L^2(\mu^1) = L^2(\mu)$, and $L_{\rm sym}^2(\mu^q)$ is the subspace of $L^2(\mu^q)$ composed of those functions that are $\mu^q$-a.e.\ symmetric; see Section \ref{sec:background}-(II). Moreover, in order to simplify the notation, we use the convention that $\|\,\cdot\,\|$ and $\langle\,\cdot\,,\,\cdot\,\rangle$ stand for the norm and the scalar product in some space $L^2(\mu^q)$ whose order $q$ will always be clear from the context.

\smallskip

Our first result is a quantitative estimate for $d_3\big(F,\cGam_\nu\big)$ in terms of the Malliavin operators $D$ and $L^{-1}$, that is, the derivative operator and the pseudo-inverse of the Ornstein-Uhlenbeck generator. We recall that the derivatives $DF$ and $DL^{-1}F$ are random elements with values in the Hilbert space $L^2(\mu)$; see Section \ref{sec:background}-(V).

\smallskip

\begin{theorem}[\textbf{General Gamma bounds}]\label{thm:SteinBound}
Let $F$ be a centred and square-integrable functional of the Poisson measure $\eta$, and assume that $F$ is in the domain of the derivative operator $D$. Then,
\begin{eqnarray}\label{eq:d2SteinBoundGeneral}
d_3(F,\cGam_\nu)&\leq& c_1A_1(F)+c_2A_2+2c_1A_3(F)\\ \notag&:=& c_1\EE\left|2(F+\nu)_+-\langle DF,-DL^{-1}F\rangle\right|+c_2\int_{\sZ}\EE[|D_zF|^2|D_zL^{-1}F|]\,\mu(\dint z)\\ \notag & & \hspace{3cm}+\,2c_1\int_{\sZ}\EE\big[(D_z{\bf 1}_{\{F>-\nu\}})(D_zF)|D_zL^{-1}F|\big]\,\mu(\dint z),
\end{eqnarray}
with constants $c_1$ and $c_2$ given by $$c_1=\max(1,1/\nu+2/\nu^2)\qquad{\rm and}\qquad c_2=\max(2/3,2/(3\nu)-3/\nu^2+4/\nu^3).$$ If in addition $\EE\big[\langle DF,-DL^{-1}F\rangle|F\big]\geq 0$ (a.s.-$\PP$), then $$A_1(F)\leq A'_1(F):= \sqrt{\EE\big[\left(2(F+\nu)-\langle DF,-DL^{-1}F\rangle\right)^2\big]},$$ and consequently
\begin{equation}\label{eq:d2SteinBoundWithoutPlus}
d_3(F,\cGam_\nu) \leq c_1A'_1(F)+c_2A_2+2c_1A_3(F)
\end{equation}
\end{theorem}

\begin{remark}\label{rem:IndicatorInDomD} {\rm 
\begin{itemize}

\item[(i)] In (\ref{eq:d2SteinBoundGeneral}), we implicitly used a `trajectorial' definition of the random function $z\mapsto D_z{\bf 1}_{\{F>-\nu\}}$, that is, we put $D_z{\bf 1}_{\{F>-\nu\}} = {\bf 1}_{\{F+D_zF>-\nu\}} - {\bf 1}_{\{F>-\nu\}}$, without necessarily assuming that $$\EE\int_{\mathcal{Z}} (D_z{\bf 1}_{\{F>-\nu\}})^2\,\mu(\dint z)<\infty$$ (note that this last relation is equivalent to the fact that ${\bf 1}_{\{F>-\nu\}}$ belongs to the set ${\rm dom}\, D$, as defined in Section 3-(V); see Lemma \ref{l:rip}). It is easily checked that $$(D_z{\bf 1}_{\{F>-\nu\}})(D_zF) = ({\bf 1}_{\{F\leq -\nu < F+D_zF\}} +{\bf 1}_{\{F+D_zF\leq -\nu <F\}})|D_zF|,$$ in such a way that $A_3(F)\geq 0$. An effective bound on $A_3(F)$, in the case where $\mu$ is a finite measure and $F$ is a multiple Wiener-It\^o integral, is presented in Proposition \ref{l:newbound}.

\item[(ii)] As first done in \cite{PSTU2010}, we shall often control the quantity $A_2(F)$ appearing in (\ref{eq:d2SteinBoundGeneral}) by using the relation
\begin{equation}\label{e:tgv}
A_2(F) \leq A_4(F)\times A_5(F) :=\left(\,\int_{\sZ}\EE[|D_zF|^4]\,\mu(\dint z)\right)^{1/2}\times  \left(\,\int_{\sZ}\EE[|D_zL^{-1}F|^2]\,\mu(\dint z)\right)^{1/2} .
\end{equation}
We also note that, if $\{F_n : n\geq 1\}$ is a sequence of random variables with bounded variances living in a fixed sum of Wiener chaoses, then the numerical sequence
$
n\mapsto A_5(F_n)
$
is necessarily bounded.
\item[(iii)] Theorem \ref{thm:SteinBound} should be compared with the following bound from \cite[Theorem 3.11]{NourdinPeccatiWienerChaos}. Let $F$ be a centered functional of a Gaussian measure on $\sZ$ with control $\mu$, and assume that $F$ is in the domain of the Malliavin derivative $D$ (see \cite[Chapter 2]{NPBook} for relevant definitions), then there exists a constant $K$ such that, for some adequate distance $d$,
$$
d(F,\cGam_\nu) \leq K\times\EE\left|2(F+\nu)_+-\langle DF,-DL^{-1}F\rangle\right|.
$$  
The presence of the additional term $$c_2\int_{\sZ}\EE[|D_zF|^2|D_zL^{-1}F|]\,\mu(\dint z)+\,2c_1\int_{\sZ}\EE\big[(D_z{\bf 1}_{\{F>-\nu\}})(D_zF)|D_zL^{-1}F|\big]\,\mu(\dint z)$$ in \eqref{eq:d2SteinBoundGeneral} or \eqref{eq:d2SteinBoundWithoutPlus} is due to the characterization of the Malliavin derivative on the Poisson space as a difference operator as well as to the non-differentiability at $-\nu$ of the solution of the Stein-equation characterizing $\cGam_\nu$; see Section \ref{sec:background}-(V). As proved in \cite[Proposition 3.9]{NourdinPeccatiWienerChaos}, on the Gaussian-Wiener space the condition $\EE\big[\langle DF,-DL^{-1}F\rangle|F\big]\geq 0$ (a.s.-$\PP$) is satisfied for every $F$ in the domain of $D$.

\item[(iv)] Other relevant one-dimensional bounds for probabilistic approximations involving Malliavin operators on the Poisson space are proved in \cite{PSTU2010}, dealing with normal approximations, \cite{PecPoisson}, dealing with the Poisson approximation of integer-valued random variables and \cite{EV}, focusing on absolutely continuous distributions whose support is given by the real line. See \cite{BourPec, PecZheng} for several multidimensional extensions. 
\end{itemize}
}
\end{remark}

As announced, we conclude the present section with a useful bound on the quantity $A_3(F)$, in the case where $F=I_q(f)$ equals a multiple Wiener-It\^o integral and the control measure $\mu$ is finite. At the cost of a heavier notation, our techniques could suitably be modified in order to deal with the case of a random variable $F$ having a finite chaotic expansion.

\begin{proposition}\label{l:newbound} Let the control measure $\mu$ be finite, and consider $F= I_q(f)$, where $q\geq 2$ and $f\in L^2_{\rm sym}(\mu^q)$. We assume that (i) $\EE\int_{\mathcal Z} (D_zF)^4\,\mu(\dint z)<\infty$, that (ii) the random function $$\sZ\ni z\mapsto D_zF|D_zF|:= v(z)$$ is such that $v(z) \in {\rm dom}\, D$ for $ \mu(\dint z)$-almost every $z$, and satisfies $$\EE\int_{\mathcal Z}\int_{\mathcal Z} (D_{z_2}v(z_1))^2\,\mu(\dint z_1)\mu(\dint z_2)<\infty.$$ Then, defining $A_3(F)$ as in (\ref{eq:d2SteinBoundGeneral}), one has the bound
\begin{equation}\label{e:art}
\begin{split}
\frac{q}{2\sqrt{2}}A_3(F) &\leq \sqrt{\EE\int_{\mathcal Z} (D_zF)^4\,\mu(\dint z)}+ \sqrt{\EE\int_{\mathcal Z}\int_{\mathcal Z} (D_{z_2}D_{z_1}F)^2(D_{z_1}F)^2\,\mu(\dint z_1)\mu(\dint z_2)}\\
& \qquad\qquad\qquad +\,\sqrt{\EE\int_{\mathcal Z}\int_{\mathcal Z} (D_{z_2}D_{z_1}F)^4\,\mu(\dint z_1)\mu(\dint z_2)}.
\end{split}
\end{equation}
\end{proposition}

\begin{remark}\rm
\begin{itemize}
\item[(i)] Another way of controlling the term $A_3(F)$, whenever $F$ has a finite chaotic expansion, is discussed in \cite{SchulteKolm}. One should note that, albeit our proof of Proposition \ref{l:newbound} also starts with an integration by parts formula, our strategy for controlling the term $A_3(F)$ is significantly different. Indeed, our approach is based on isometric formulae for divergence operators, whereas \cite{SchulteKolm} uses a direct estimation consisting in controlling $|DF|$ by a random function having a finite chaotic expansion. When applied to our framework in the case $q>2$, the technique used in \cite{SchulteKolm} leads to expressions involving contractions of the absolute value of the kernel $f$, therefore producing bounds that are systematically larger than ours. When applied to the case $q=2$, the strategy adopted in \cite{SchulteKolm} leads to slower rates of convergence, but allows in principle to dispense with the assumption that the underlying cont
 rol measure has finite mass. Since all 
our applications concern sequences of control measures having a finite mass, and for the sake of conciseness, we will omit a formal discussion of this fact.

\item[(ii)] From the standpoint of geometric applications, focusing on Poisson measures having a finite control is barely a restriction. Indeed, the kind of geometric limit theorems we are interested in typically involve either functionals of a Poisson measure having a finite control, whose total mass asymptotically explodes (like the ones we consider in the applications developed later in the paper), or functionals of the restriction of a Poisson measure to a finite window with growing volume; see e.g.\ \cite{BourPec, DFRV,LRP,LRP2,LPST, PecPoisson, ReitznerSchulte,ST12} for a recent collection of distinguished examples.
\end{itemize}
\end{remark}

\subsection{Simplified estimates for supports contained in a half-line}

The applications we are interested in require that the we consider random variables possibly taking values in the half-line $(-\infty, -\nu)$, in such a way that the rather unusual term $A_3(F)$ cannot be dispensed with. However, if one is only interested in measuring the distance between $\cGam_\nu$ and the law of a random variable with support in $[-\nu, +\infty)$, then the statement of Theorem \ref{thm:SteinBound} can be significantly simplified, since in this case the term $A_3(F)$ disappears. In particular, whenever the law of $F$ satisfies these requirements, the finiteness of the measure $\mu$ does not play any role. This point is made clear in the next statement whose easy proof is left to the reader.

\begin{proposition} Let $F$ be a centered square-integrable functional of the random measure $\eta$. Assume that the law of $F$ has support in $[-\nu, +\infty)$ and that $F$ is in the domain of the derivative operator $D$. Then, the bound (\ref{eq:d2SteinBoundGeneral}) holds with $A_3(F) =0$. If moreover $\EE\big[\langle DF,-DL^{-1}F\rangle|F\big]\geq 0$ (a.s.-$\PP$), then the estimate (\ref{eq:d2SteinBoundWithoutPlus}) holds with $A_3(F) =0$.
\end{proposition}

\subsection{General results for sequences of multiple integrals}

We now focus on the following setup. Let $(\mathcal{Z}, \mathscr{Z})$ be a fixed Polish space as above, and $\{\eta_n : n\geq 1\}$ be a sequence of Poisson random measures on $(\mathcal{Z}, \mathscr{Z})$, such that, for each $n$, the non-atomic control measure $\mu_n$ of $\eta_n$ is finite. In view of applications, we allow that $\mu_n(\mathcal{Z})\to\infty$, as $n\to\infty$. For a given even integer $q\geq 2$, we consider a sequence $\{I_q(f_n) : n\geq 1 \}$ of multiple Wiener-It\^o stochastic integrals with the following characteristics: (a) $\{f_n:n\geq 1\}\subset L^2_{\rm sym}(\mu_n^q)$ is composed of kernels satisfying the technical assumptions stated in Section \ref{sec:background}-(VIII) below, and (b) for every $n\geq 1$, the integral $I_q(f_n)$ is realized with respect to the compensated Poisson measure $\hat{\eta}_n =\eta_n-\mu_n$. The next theorem characterizes the convergence of the distribution of $I_q(f_n)$, as $n\to\infty$, to the limit law $\cGam_\nu$. The set
  of analytic conditions appearing below is expressed in 
terms of (possibly symmetrized) contraction kernels, whose definition is provided in Section \ref{sec:background}-(VI). Observe in particular that $f_n \star_q^0 f_n = f_n^2$. 
\smallskip

\begin{theorem}[\textbf{Gamma limits in the Poisson-Wiener chaos}]\label{thm:GammaContractions} Let the above assumptions and notation prevail (in particular, $\mu_n$ is a finite measure for every $n$), let $q\geq 2$ be an even integer and let $\{f_n:n\geq 1\}\subset L^2_{\rm sym}(\mu_n^q)$ be such that $\lim\limits_{n\to\infty}q!\|f_n\|^2=2\nu$, and suppose that the technical conditions of Section \ref{sec:background}-(VIII) are satisfied. Assume in addition that
\begin{equation}\label{eq:GammaContractionsCondition}
\lim_{n\to\infty}\|f_n\s_r^\ell f_n\|=0\quad{and}\quad\lim_{n\to\infty}\|f_n\ts_{q/2}^{q/2}f_n-c_qf_n\|=0\quad{with}\quad c_q={4\over \left({q\over 2}\right)!{q\choose q/2}^2}
\end{equation}
for all pairs $(r,\ell)$ such that either $r=q$ and $\ell=0$, or $r\in\{1,\ldots,q\}$, $\ell\in\{1,\ldots,\min(r,q-1)\}$ and $r$ and $\ell$ are not equal to $q/2$ at the same time. Then, the distribution of $I_q(f_n)$ converges to $\cGam_\nu$ as $n\to\infty$. Moreover, for some positive finite constant $K$ independent of $n$,
\begin{eqnarray}\notag
d_3(I_q(f_n),\cGam_\nu) &\leq& c_1A_1(I_q(f_n))+c_2A_4(I_q(f_n))\times A_5(I_q(f_n)) + 2c_1A_3(I_q(f_n))\\ &\leq & K\times \max\big\{\big| q!\|f_n \|^2-2\nu\big|; \|f_n\s_p^p f_n\|; \|f_n\s_r^\ell f_n\|^{1/2}; \|f_n\ts_{q/2}^{q/2}f_n-c_qf_n\|\big\}\to 0,\label{e:estimate}
\end{eqnarray}
where we have used the notation introduced in (\ref{eq:d2SteinBoundGeneral})--(\ref{e:tgv}), and the maximum is taken over all $p=1,\ldots,q-1$ such that $p\neq q/2$ and all $(r,\ell)$ such that $r\neq \ell$ and either $r=q$ and $\ell=0$, or $r\in\{1,\ldots,q\}$ and $\ell\in\{1,\ldots,\min(r,q-1)\}$. \end{theorem}

\begin{example}{\rm 
\begin{itemize}
\item[(i)] Assume $q=2$. Then, $c_2=1$ and the maximum in (\ref{e:estimate}) is taken over the following four quantities:
$$
\big|2\|f_n \|^2-2\nu\big|, \,\,\,\, \|f_n\s_2^0 f_n\|^{1/2}, \,\,\,\, \|f_n\s_2^1 f_n\|^{1/2}, \,\,\,\, \|f_n\ts_{1}^{1}f_n-f_n\|.
$$

\item[(ii)] Assume $q=4$. Then, $c_4=1/18$ and the maximum in (\ref{e:estimate}) is taken over the following ten quantities:
\begin{eqnarray*}
&& \big|2\|f_n \|^2-2\nu\big|, \,\,\,\, \|f_n\s_1^1 f_n\|, \,\,\,\, \|f_n\s_4^0 f_n\|^{1/2} ,\,\,\,\,  \|f_n\s_4^1 f_n\|^{1/2},\,\,\,\, \|f_n\s_4^2 f_n\|^{1/2}, \,\,\,\, \|f_n\s_4^3 f_n\|^{1/2}, \\
&& \|f_n\s_3^1 f_n\|^{1/2},\,\,\,\, \|f_n\s_3^2 f_n\|^{1/2},\,\,\,\,\|f_n\s_2^1 f_n\|^{1/2}, \,\,\,\, \|f_n\ts_{1}^{1}f_n-18^{-1}f_n\|, 
\end{eqnarray*}
where we have used the fact that $\|f_n\s_1^1 f_n\| = \|f_n\s_3^3 f_n\|$.
\end{itemize}
}
\end{example}

\medskip

\begin{remark} {\rm
\begin{itemize}

\item[(i)] Under the assumptions in the statement, one has that the sequence $$ A_5(I_q(f_n)) := \left(\,\int_{\sZ}\EE[|D_zL^{-1}I_q(f_n)|^2]\,\mu_n(\dint z)\right)^{1/2}$$ is such that
\begin{equation*}
A_5(I_q(f_n))^2 = (q-1)!\|f_n\|^2\to \frac{2\nu}{q}>0\qquad{\rm as}\qquad n\to\infty.
\end{equation*}
It follows that our inequality (\ref{e:estimate}) not only provides an analytic bound in the distance $d_3$, but also ensures that the three numerical sequences $\{A_1(I_q(f_n)): n\geq 1\}$, $\{A_3(I_q(f_n)): n\geq 1\}$ and $\{A_4(I_q(f_n)): n\geq 1\}$ (all related to Malliavin operators) converge to zero. This fact is crucial when dealing with the multidimensional results discussed in Section \ref{ss:multi}. An analogous remark applies to Proposition \ref{p:wi2} and Theorem \ref{t:gdj} below.

 \item[(ii)] Similar conditions (only involving contractions of the type $\star_r^r$, with $r=1,\ldots,q-1$) in the case of multiple integrals with respect to a Gaussian measure can be found in \cite[Theorem 1.2]{NourdinPeccatiNoncentral}. Non-central results of a similar flavor, in the context of free probability and multiple integrals with respect to a free Brownian motion, are proved in \cite{NPfree}.
 \item[(iii)] We were able to deduce meaningful conditions for Gamma approximations only in the case of an {\it even} integer $q\geq 2$. However, unlike in the Gaussian case (see \cite[Remark 1.3]{NourdinPeccatiNoncentral}), in a Poisson framework one cannot exclude a priori the existence of a sequence of multiple integrals of odd order converging to a limiting Gamma distribution. We prefer to consider this issue as a separate problem, and keep it as an open direction for future research.
 \item[(iv)] In the estimate (\ref{e:estimate}), and in contrast to the main bounds on normal approximations proved in \cite{PSTU2010}, norms of the type $\|f_n\s_r^\ell f_n\|$, $r\neq \ell$, appear under a square root. This phenomenon seems unavoidable, and it is directly related to the presence of cross terms arising from the specific form of the Stein equation associated with the Gamma distribution.
\end{itemize}}
\end{remark}

The following statement shows that condition (\ref{eq:GammaContractionsCondition}) might take a particularly attractive form in the case of double Poisson integrals. This will be used in order to prove the results presented in Section \ref{ss:dejong}, dealing with the Gamma approximation of degenerate $U$-statistics. 

\smallskip

\begin{proposition}[\textbf{Three moments suffice for Gamma approximations}]\label{p:wi2} Let the control measures $\{\mu_n : n\geq 1\}$ be finite, let $q = 2$ and let $\{f_n:n\geq 1\}\subset {L^2_{\rm sym}(\mu_n^2)}$ be such that $\lim\limits_{n\to\infty}\EE[I^2_2(f_n)] = \lim\limits_{n\to\infty}2\|f_n\|^2=2\nu$, and such that the technical conditions of Section \ref{sec:background}-(VIII) are satisfied. Assume in addition that $\int_{\sZ}f_n^4\,\dint\mu_n^2 \to 0$ and that $\EE[I_2^4(f_n)] <\infty$ for every $n$. Then, condition (\ref{eq:GammaContractionsCondition}) is verified if and only if 
\begin{equation}\label{e:moments}
\EE[I_2^4(f_n)] -12\EE[I_2^3(f_n)] \longrightarrow 12\nu^2- 48\nu\qquad{\rm as}\qquad n\to\infty. 
\end{equation}
In particular, if the sequence $F_n^4$ is uniformly integrable, then (\ref{eq:GammaContractionsCondition}) and (\ref{e:moments}) are both necessary and sufficient in order to have that the distribution of $F_n$ converges to $\cGam_\nu$ in the sense of the distance $d_3$.
\end{proposition}

\subsection{An extension of de Jong's theorem for degenerate U-statistics}\label{ss:dejong}

In the present and the subsequent section, we shall work within the following framework. We fix an integer $d\geq 1$, and let ${\bf Y} = \{Y_i : i\geq 1\}$ be a sequence composed of i.i.d.\ random variables with values in $\RR^d$, whose common distribution has a density $p(x)$ with respect to the Lebesgue measure on $\RR^d$ (written $\dint x$). The sequence $\{N(n) : n\geq 1\}$ of integer-valued random variables is independent of ${\bf Y}$ and such that, for every $n$, $N(n)$ has a Poisson distribution with parameter $n$. It is well-known that, in this setting, the random point measure 
\begin{equation}\label{e:etan}
\eta_n := \sum_{i=1}^{N(n)} \delta_{Y_i}
\end{equation}
(where $\delta_y$ represents the Dirac mass at $y$) is a Poisson measure on $\mathcal{Z} = \RR^d$ (equipped with the standard Borel $\sigma$-field $\mathscr{B}(\RR)^{\otimes d}$) with control measure $\mu_n(\dint x) =np(x)\dint x$. We shall also use the shorthand notation $\mu(\dint x) :=\mu_1(\dint x) = p(x)\dint x$. 

\smallskip

Our aim below is to provide a Gamma-type counterpart to a famous theorem by P. de Jong, proved in \cite{DJ}, involving sequences of degenerate $U$-statistics of order 2. We stress that the results contained in \cite{DJ} have later been extended to degenerate $U$-statistics of a general order; see \cite{BhGh,DJMulti}. Albeit our method clearly applies to these general objects, we prefer here to focus on $U$-statistics of order 2, in order to obtain neater statements and to emphasize the method over technical details. We start with some useful definitions.

\begin{definition}[\textbf{$U$-statistics}]\label{d:degu}{\rm 
\begin{itemize}
 \item[(i)] Let $k\geq 2$, and let $h : \RR^q \to \RR$ be a symmetric kernel such that $ h \in L^1_{\rm sym} (\mu^k)$. The (symmetric) {\it $U$-statistic of order $k$} based on $h$ and on the sample $\{Y_1,\ldots,Y_{m}\}$ (where $m\geq k$ is some integer) is the random variable
\begin{equation}\label{d:us} 
U_m(h,{\bf Y}) =\sumnod h(Y_{i_1},\ldots,Y_{i_k}), 
\end{equation}
where the symbol $\sum^{\neq}$ indicates that the sum is taken over all vectors $(i_1,\ldots,i_k)$ such that $i_j \neq i_\ell$ for every $j \neq \ell$.
 \item[(ii)] Fix $k\geq 2$ and let $U_m(h,{\bf Y})$ be a symmetric $U$-statistic as in (\ref{d:us}). The {\it Hoeffding rank} of $U_m(h,{\bf Y})$ is the smallest integer $1\leq q\leq k$ such that $ \EE[h(Y_1,\ldots,Y_k) | Y_1,\ldots,Y_{q-1}]  = 0$ (a.s.-$\mathbb{P}$) and  $ \EE[h(Y_1,\ldots,Y_k) | Y_1,\ldots,Y_{q}] \neq 0$, where $\EE[h(Y_1,\ldots,Y_k) | Y_1,\ldots,Y_{0}]  := \EE[h(Y_1,\ldots,Y_k)]$. A $U$-statistic of order $k$ with Hoeffding rank equal to $k$ is said to be {\it completely degenerate}. In other words, a $U$-statistic such as (\ref{d:us}) is completely degenerate if $h$ is a non-zero kernel verifying
$$\int_\RR h(x,y_1,\ldots,y_{k-1}) p(x)\, \dint x = 0\qquad (\mu^{k-1}-\mbox{a.e.}).$$
\item[(iii)] A collection of random variables $\{F_n : n\geq 1\}$ is said to be a sequence of {\it geometric $U$-statistics} of order $k$, if there exists a kernel $ h \in L^1_{\rm sym} (\mu^k)$ such that 
\begin{equation*}\label{e:gus}
F_n = U_{N(n)}(h,{\bf Y}), \quad n\geq 1,
\end{equation*}
where $\{N(n) : n\geq 1\}$ is the independent Poisson sequence introduced above.
\end{itemize}}
\end{definition}

Before presenting the main result of this section, and in order to make the connection with our general framework more transparent, we shall recall an important finding from \cite[Lemma 3.5 and Theorem 3.6]{ReitznerSchulte}, stating that Poissonized $U$-statistics of order $k$ live inside the sum of the first $k+1$ Wiener chaoses associated with the Poisson measure $\eta_n$. The proof heavily relies on results by Last and Penrose \cite{LastPenrose}.

\smallskip

\begin{lemma}[\textbf{Reitzner and Schulte}]
\label{prop:L1L2}
Consider a kernel $h\in L_{\rm sym}^1(\mu^k)$ such that the corresponding Poissonized $U$-statistic $U_{N(n)}(h, {\bf Y})$ is square-integrable. Then, $h$ is necessarily in $L_{\rm sym}^2(\mu^k)$, and $U_{N(n)}(h, {\bf Y})$ admits a chaotic representation of the type
$$
U_{N(n)}(h, {\bf Y}) = \EE[U_{N(n)}(h, {\bf Y})] +\sum_{i=1}^k n^{k-i} I_i (h_i)
$$ 
where $I_i$ indicates a multiple Wiener-It\^o integral of order $i$ with respect to the compensated Poisson measure $\hat\eta_n = \eta_n - \mu_n$, defined according to (\ref{e:etan}), and
\begin{equation}\label{e:lastpenrose}
h_i (z_1,\ldots,z_i)= \binom{k}{i}\int_{\sZ^{k-i}} h(z_1,\ldots,z_i, \bullet )\, \mu^{k-i}(\dint\,\bullet),\quad (z_1,\ldots,z_i) \in \sZ^i,
\end{equation}
where the bullet ``$\, \bullet$'' stands for a packet of $k-i$ variables that are integrated with respect to $\mu^{k-i}$. In particular, $h=h_k$ and the projection $h_{i}$ is in $ L_{\rm sym}^{2}(\mu^i)$ for each $1 \leq i \leq k$.\\
\end{lemma}

The following statement corresponds to the main result proved by de Jong in \cite{DJ}, in the special case of symmetric $U$-statistics of order $2$ (note that the assumption that the underlying kernels have finite moments of order four is only implicit in de Jong's work). Given positive sequences $a_n, b_n$, $n\geq 1$, we write $a_n\approx b_n$ whenever $\lim\limits_{n\to\infty} a_n/b_n=1$.

\smallskip

\begin{theorem}[\textbf{de Jong}]\label{t:dj} Let $\{h_n : n\geq 1\}$  be a sequence of non-zero elements of  $L^4_{\rm sym} (\mu^2)$. Define $F_n = U_n(h_n, {\bf Y})$ and assume that $F_n$ is completely degenerate. Then, one has that $\sigma^2(n):= {\rm Var}( F_n) \approx 2n^2 \EE[h_n(Y_1,Y_2)^2]$, and the fourth moment condition
$$
\lim_{n\to\infty}\frac{\EE[F_n^4]}{\sigma(n)^4} = 0,
$$
implies that, as $n\to \infty$, the sequence $\widetilde{F}_n := F_n/\sigma(n)$ converges in distribution to a standard Gaussian random variable.
\end{theorem}

The following statement consists of two parts. Part (A) is a quantitative extension of Theorem \ref{t:dj} based on a direct study of the fourth moments of the Poissonized $U$-statistic, whereas part (B) is a Gamma-type extension of de Jong's theorem which is directly based on the results discussed in Section \ref{ss:gen}. Apart from \cite{DJ}, our findings should be compared with the seminal work by Jammalamadaka and Janson \cite{JJ}, about the normal and Poisson approximation of $U$-statistics of order two. To our knowledge, the forthcoming Theorem \ref{t:gdj} is the first quantitative extensions of the de Jong theorem, also dealing with the non-normal approximation of general degenerate $U$-statistics. Moreover, we would like to emphasize that our proof of Part (A) is shorter and more transparent than the one presented in the original work \cite{DJ} (one should note that, however, our methods only allow us to deal with symmetric $U$-statistics). Recall that the {\it Wasserstein distance} between the laws 
of two integrable random variables $X,Y$ is given by
$$
d_W(X,Y) := \sup_{h\in {\rm Lip}(1)} \left| \EE[h(X)] - \EE[h(Y)]\right|,
$$
where ${\rm Lip}(1)$ is the set of Lipschitz functions $h:\RR\to\RR$ with a Lipschitz constant $\leq 1$. Recall that, in the framework of this section, $\mathcal{Z} = \RR^d$.
\smallskip

\begin{theorem}[\textbf{Extended de Jong theorem}]\label{t:gdj} Let $\{h_n : n\geq 1\}$  be a sequence of non-zero elements of  $L^4_{\rm sym} (\mu^2)$ such that 
\begin{equation*}\label{e:sup}
\sup_n\,  \frac{\int_{\sZ}h_n^4\,\dint\mu_n^2}{\left( \int_{\sZ}h_n^2\,\dint\mu_n^2\right)^2} <\infty.
\end{equation*}
Put $F_n = U_n(h_n, {\bf Y})$ and $F'_n = U_{N(n)}(h_n, {\bf Y})$, and assume that these $U$-statistics are completely degenerate. Then, $\sigma(n)^2:= {\rm Var}( F_n) \approx {\rm Var}( F'_n)  = 2n^2\, \EE[h_n(Y_1,Y_2)^2]$, and the following two points {\rm (A)} and {\rm (B)} hold.
\begin{itemize}
\item[\rm (A)] If
\begin{equation}
\label{e:c1}\frac{\EE[(F'_n)^4]}{\sigma(n)^4} \to 0\qquad{\rm as}\qquad n\to\infty,
\end{equation}
then both $\widetilde{F}_n := F_n/\sigma(n)$ and $\widetilde{F}'_n := F'_n/\sigma(n)$ converge in distribution to a standard Gaussian random variable $N$. Moreover, there exists a universal finite constant $K$, independent of $n$, such that, as $n\to\infty$,
\begin{eqnarray}
\label{e:e1}  d_W(\widetilde{F}'_n , N) \leq K\times B_n &\longrightarrow& 0\,,\\
\label{e:e2} d_W(\widetilde{F}_n , N) \leq K\times\big(B_n+n^{-1/4}\big)&\longrightarrow& 0,
\end{eqnarray}
with $B_n:=\sigma(n)^{-2}\max\left\{\big(\int_{\sZ}h_n^4\,\dint\mu_n^2\big)^{1/2}; \|h_n\star_1^1 h_n \|; \|h_n\star_2^1 h_n \|\right\}$.

\item[\rm (B)] If $\int_{\sZ}h_n^4\,\dint\mu_n^2\to 0$ and there exists $\nu>0$ such that $\sigma(n)^2\to 2\nu$, and 
\begin{equation}\label{e:c2}
\EE[(F_n')^4] -12\EE[(F_n')^3] \longrightarrow 12\nu^2- 48\nu\qquad{\rm as}\qquad n\to\infty,
\end{equation}
then both $F_n$ and $F'_n$ converge in distribution to a random variable $G(\nu)$, which has distribution $\cGam_\nu$. Moreover, there exists a universal constant $K>0$ such that, as $n\to \infty$,
\begin{eqnarray}
\label{e:e3} d_3(F'_n,\cGam_\nu) \leq c_1A_1(F'_n)+c_2A_4(F'_n)\times A_5(F'_n)\leq K \times C_n & \longrightarrow& 0\,, \\
 \label{e:e4}d_3({F}_n , \cGam_\nu) \leq K\times \big( C_n +n^{-1/4}\big)&\longrightarrow& 0,
\end{eqnarray}
with $C_n:=\max\left\{ | 2\|h_n\|^2-2\nu|; \big(\int_{\sZ}h_n^4\,\dint\mu_n^2\big)^{1/4};\|h_n\s_2^1 h_n\|^{1/2}; \|h_n\ts_{1}^{1}h_n-h_n\|\right\}$, and we have used the notation introduced in (\ref{eq:d2SteinBoundGeneral})--(\ref{e:tgv}).
\end{itemize} 
\end{theorem}

\smallskip

\begin{remark}\label{r:lapalisse}{\rm Our proof of Theorem \ref{t:gdj} shows indeed that the quantity $B_n$ (reps.\ $C_n$) in the statement converges to zero if and only if the asymptotic condition (\ref{e:c1}) (resp.\ \eqref{e:c2}) is verified. }
\end{remark}

\subsection{Gamma convergence of geometric $U$-statistics: characterization and bounds}

As anticipated, the aim of this section is to apply the main estimates of the present paper in order to characterize the class of geometric $U$-statistics based on ${\bf Y}$ converging in distribution towards a Gamma random variable. Since our analysis is based on Theorem \ref{t:gdj}, our results will provide explicit estimates on the speed of convergence. We refer the reader to \cite{DyMa, RubVit} for some classic references on the subject and to \cite{LRP2, ReitznerSchulte} for a discussion of several recent developments. We let the notation and assumptions of the previous section prevail and recall that a {\it Gaussian measure} $G$ on $\big(\RR^d, \mathscr{B}(\RR)^{\otimes d}\big)$, with control $\mu(\dint x) = p(x)\dint x$, is a centred Gaussian family of the type
$$
G = \{G(B) : B\in \mathscr{B}(\RR)^{\otimes d}, \,\, \mu(B)<\infty\}
$$
such that, for every $m \geq 1$ and every $B_1,\ldots,B_m\in\mathscr{B}(\RR)^{\otimes d}$ with $\mu(B_i)<\infty$ ($i=1,\ldots,m$), the vector $\big(G(B_1),\ldots,G(B_m)\big)$ has an $m$-dimensional joint Gaussian distribution with covariance matrix $\EE[G(B_i)G(B_j)] = \mu(B_i\cap B_j)$.

\smallskip

The next statement combines findings from \cite[Section 7]{LRP2} (point (i)) with a classic characterization of elements in the second Wiener chaos of a Gaussian measure (point (ii); see \cite[Section 2.7.4]{NPBook} for more details.

\smallskip

\begin{proposition}\label{p:LR} Let $k\geq 2$ and let $h\in L^1_{\rm sym}(\mu^k)$ be a non-zero kernel such that the $U$-statistic $F'_n := U_{N(n)}(h,{\bf Y})$ is square-integrable for every $n$, and has Hoeffding rank equal to 2. For $n\geq 1$, define also the standardized $U$-statistic $\widetilde{F}'_n = n^{1-k} F'_n$.
\begin{itemize}
\item[\rm (i)] For every $n$, there exists a sequence of double integrals $I_2(f_n)$ (each realized with respect to the compensated Poisson measure $\eta_n - \mu_n$) such that, as $n\to \infty$, $\EE[(\widetilde{F}'_n-I_2(f_n))^2]\to 0$. Moreover, $\widetilde{F}'_n$ converge in distribution to $I_2^G(h_2)$, where $I_2^G$ indicates a double Wiener-It\^o integral with respect to the Gaussian measure $G$, and $h_2$ is defined according to (\ref{e:lastpenrose}). The same convergence takes place for the de-Poissonized $U$-statistics $\widetilde{F}_n = n^{1-k} F_n$, where $F_n := U_{n}(h,{\bf Y})$.
\item[\rm (ii)] The random variable $I_2^G(h_2)$ cannot be Gaussian. Moreover, assume that $I_2^G(h_2)$ follows a $\cGam_\nu$-distribution. Then, necessarily, $\nu\in \{1,2,\ldots\}$ and there exists an orthonormal system $\{e_1,\ldots,e_\nu\}\subset L^2(\mu)$ such that $\displaystyle h_2 = \sum_{i=1}^\nu e_i\otimes e_i$ and $\EE[e_i(Y_1)] = 0$ for $i=1,\ldots,\nu$.  
\end{itemize}
\end{proposition}

\begin{remark} {\rm Let $k=2$, and consider $h_2=\sum\limits_{i=1}^\nu e_i\otimes e_i$, as in the statement of Proposition \ref{p:LR}-(ii). Then, it is easily seen (by a direct computation) that 
$$
U_{n}(h_2,{\bf Y}) = \sum_{i=1}^\nu\left[\left(\sum_{k=1}^n e_i(Y_k) \right)^2 - \sum_{k=1}^n e_i(Y_k)^2\right].
$$
The fact that the distributions of $\widetilde{F}'_n $ and $\widetilde{F}_n $ converge to $\cGam_\nu$ is therefore a direct consequence of the usual multidimensional central limit theorem and of the law of large numbers. }
\end{remark}
\smallskip
The next statement is a quantitative counterpart to Proposition \ref{p:LR}-(ii), containing in particular estimates involving Malliavin operators. Such estimates will be put into use in the Examples \ref{ex:lilly}--\ref{ex:lilly4} below, where the asymptotic behavior of a $U$-statistic such as $U_{N(n)}(h_2,{\bf Y})$ is studied within the framework of hybrid convergence in random graphs, random flat and random simplex models.

\smallskip

\begin{theorem}[\textbf{Bounds on Gamma convergence}]\label{t:gammaU} Let the assumptions and notation of Proposition \ref{p:LR} prevail. Assume moreover that $I_2^G(h_2)$ has distribution $\cGam_\nu$ for some $\nu=1,2,\ldots$, and also that $\{e_1,\ldots,e_\nu\}\subset L^4(\mu)$, where the orthonormal system $\{e_1,\ldots,e_\nu\}$ is defined in Proposition \ref{p:LR}-(ii). Then, there exists a finite constant $K$, independent of $n$, such that 
\begin{eqnarray*}\label{e:sss}
d_3(\widetilde{F}'_n, \cGam_\nu)&\leq& c_1A_1(\tilde{F}'_n)+c_2A_4(\tilde{F}'_n)\times A_5(\tilde{F}'_n) \leq K\times {n^{-1/4}}, \quad\\ 
\quad d_3(\widetilde{F}_n, \cGam_\nu)   &\leq& K\times {n^{-1/4}},
\end{eqnarray*}
where in the first inequality we used the notation defined in (\ref{eq:d2SteinBoundGeneral})--(\ref{e:tgv})
\end{theorem}

\smallskip

\begin{example}{\rm 
\begin{itemize}

\item[(i)]\label{ex:rs1} Let $g_1, g_2$ be two orthonormal elements of $L^2(\mu)$ such that $g_1, g_2\in L^4(\mu)$ and $\int_{\mathcal{Z}} g_1(z)\,\mu(\dint z) = \int_{\mathcal{Z}} g_2(z)\,\mu(\dint z)$. We stress that we do not require that $g_1,g_2$ have disjoint supports. Then, the kernel
\begin{equation*}
h_2(z_1,z_2) = \frac{1}{2}(g_1-g_2)\otimes (g_1-g_2)\,  (z_1,z_2) =\frac{g_1(z_1)-g_2(z_1)}{\sqrt{2}}\times \frac{g_1(z_2)-g_2(z_2)}{\sqrt{2}} 
\end{equation*}
is such that the corresponding $U$-statistics of order two $F_n := U_{n}(h_2,{\bf Y})$ and $F'_n := U_{N(n)}(h_2,{\bf Y})$ are completely degenerate, and both converge in distribution to $\cGam_1$, with an upper bound of order $n^{-1/4}$ on the rate of convergence.

\item[(ii)] As an example of a pair $(g_1,g_2)$ verifying the requirements at Point (i), one can take $g_1 = \sqrt{2}\,{\bf 1}_{A}$ and $g_2 = \sqrt{2}\,{\bf 1}_{B}$, where $\{A,B\}$ is a measurable partition of $\mathcal{Z}$ such that $\mu(A) = \mu(B) = 1/2$. Considering the case $d=1$, $p(\,\cdot\,) = \frac12\, {\bf 1}_{(-1,1)}(\,\cdot\,)$, $g_1(z) = \sqrt{2}\, {\bf 1}_{(0,1)}(z)$ and $g_2(z) = \sqrt{2}\, {\bf 1}_{(-1, 0)}(z)$, one obtains a kernel $h_2$ with support in $(-1,1)^2\backslash \{(0,0)\}$ and such that $h_2(z_1,z_2) = 1$ if $z_1z_2>0$, and $h_2(z_1,z_2) = -1$ if $z_1z_2<0$. In this way, one recovers the non-central result discussed by Reitzner and Schulte in \cite[end of Section 5.1]{ReitznerSchulte}.
\end{itemize}
}
\end{example}

\subsection{Multivariate extensions and hybrid convergence}
\label{ss:multi}
We describe here three multivariate extensions of the results in Section
\ref{ss:gen}. The first two results can be seen as partial analogues on the Poisson space of
\cite[Theorem 4.4]{NR} -- for the multivariate Gamma convergence -- and \cite[Theorem 4.5]{NR}
-- for the hybrid convergence -- both concerning sequences of
multiple integrals with respect to a Gaussian measure. Observe that the method used in \cite{NR}  is based on new criteria for asymptotic independence of
multiple integrals. These criteria are not available on the Poisson space. For this reason, our approach is different and will be in the spirit of the ``interpolation method'' used in \cite{BourPec}. Such an interpolation method will be also used to deduce our third result, concerning hybrid Poisson/Gamma convergence. As already pointed out, the general problem of characterizing independence is rather well-understood in a Gaussian
framework (cf.\ \cite{Kallenberg,UstunelZakai89,UstunelZakai90}), while the topic is still largely open in the context Poisson measures; see \cite{Privault, RosSam}. 

\smallskip

\begin{remark}\label{r:co}{\rm As before, we consider the framework of a sequence of Poisson measures $\{\eta_n : n\geq 1\}$ (on some Polish space $(\mathcal{Z},\mathscr{Z})$), each having a finite with non-atomic control measure $\mu_n$. For the entire section, $I_q$ denotes the multiple Wiener-It\^o integral, of order $q$, with respect to one of the compensated measures $\hat{\eta}_n = \eta_n - \mu_n$ (the concerned index $n$ will always coincide with the index of the integrated function, for instance: $I_q(f_n)$ indicates the multiple integral of order $q$ of $f_n$ with respect to $\hat\eta_n$).
}
\end{remark}

\smallskip

Let $d\geq 1$ be a fixed integer, let $\nu_1,\ldots,\nu_d>0$ and let
$(G_1,\ldots,G_d)$ be a vector consisting of independent random variables
such that $G_i$ has the centred Gamma distribution $\cGam_{\nu_i}$.
Further let $2\leq q_1<q_2<\ldots<q_d$ be even integers satisfying
$2q_i\neq q_j$ for any $i\neq j$ and let for $i\in\{1,\ldots,d\}$,
$\{f_n^{(i)}:n\geq 1\}$ be a sequence of kernels such that $f_n^{(i)} \in L^2_{\rm sym}(\mu_n^{q_i})$, satisfying in addition the technical conditions of Section \ref{sec:background}-(VIII). The next result deals with the announced multivariate Gamma
convergence. We emphasize that we do not need further conditions on
asymptotic covariances due to our assumption that all multiple integrals
have different orders.

\smallskip

\begin{theorem}[{\bf Multivariate Gamma convergence}]\label{thm:MultivariateGamma}
Let the above notation and conditions prevail. For any
$i\in\{1,\ldots,d\}$ assume that
$\lim\limits_{n\to\infty}q_i!\|f^{(i)}_n\|^2=2\nu_i$ and that
\begin{equation*}\label{eq:GammaContractionsConditionMultivariate}
\lim_{n\to\infty}\|f_n^{(i)}\s_{r_i}^{\ell_i}
f_n^{(i)}\|=0\quad{and}\quad\lim_{n\to\infty}\|f_n^{(i)}\ts_{q_i/2}^{q_i/2}f_n^{(i)}-c_{q_i}f_n^{(i)}\|=0\quad{with}\quad
c_{q_i}={4\over \left({q_i\over 2}\right)!{q_i\choose q_i/2}^2}
\end{equation*}
for all pairs $(r_i,\ell_i)$ such that either $r_i=q_i$ and $\ell_i=0$, or $r_i\in\{1,\ldots,q_i\}$, $\ell_i\in\{1,\ldots,\min(r_i,q_i-1)\}$
and $r_i$ and $\ell_i$ not equal to $q_i/2$ at the same time. Then,
$\big(I_{q_1}(f_n^{(1)}),\ldots,I_{q_d}(f_n^{(d)})\big)$ converges in
distribution to $\big(G_1,\ldots,G_d\big)$ as $n\to\infty$.
\end{theorem}

We go one step further and turn to an extension of Theorem
\ref{thm:MultivariateGamma} where we consider convergence of a random
vector of multiple integrals to a hybrid random vector whose components
are independent and in part centered Gamma and in part standard Gaussian
random variables. A similar setting with Poisson random variables instead
of centred Gamma ones has recently been studied in \cite{BourPec}. However, we would like to emphasize that, in contrast to \cite{BourPec}, here both distributions considered in the target vector are absolutely continuous with respect to the Lebesgue measure on the real line.

\smallskip

To formulate our result on the Gamma/Gaussian hybrid convergence, let
$d_1,d_2\geq 1$ be fixed integers, let $\nu_1,\ldots,\nu_{d_1}>0$ and let
$(G_1,\ldots,G_{d_1},N_{d_1+1},\ldots,N_{d_1+d_2})$ be a vector consisting
of independent random variables such that $G_i$ has distribution
$\cGam_{\nu_i}$ for $i\in\{1,\ldots,d_1\}$ and $N_i$ has a standard
Gaussian distribution for $i\in\{d_1+1,\ldots,d_1+d_2\}$. Further let
$2\leq q_1,q_2,\ldots,q_{d_1+d_2}$ be integers such that the following constraints are verified: (a) $q_1<\cdots <q_{d_1}$ and, in general, $q_i\neq q_j$ for every $1\leq i\neq j\leq d_1+d_2$, (b) $q_1,\ldots,q_{d_1}$ are even integers, and (c) there is no pair $(i,j) \in \{1,\ldots,d_1+d_2\}^2$ such that $i\in \{1,\ldots,d_1\}$ and $2q_i = q_j$. For
$i\in\{1,\ldots,d_1+d_2\}$, we let $\{f_n^{(i)}:n\geq 1\}$ be a sequence of
symmetric and square-integrable kernels such that $f_n^{(i)}\in L^2_{\rm sym}(\mu_n^{q_i})$, satisfying moreover the technical conditions stated in Section \ref{sec:background}-(VIII).

\smallskip

\begin{theorem}[{\bf Gamma/Normal hybrid convergence}]\label{thm:HybridConvergence}
Let the above notation and conditions prevail. Assume that
$\lim\limits_{n\to\infty}q!\|f^{(i)}_n\|^2=2\nu_i$ whenever
$i\in\{1,\ldots,d_1\}$ and
that $\lim\limits_{n\to\infty}q_i!\|f^{(i)}_n\|^2=1$ whenever
$i\in\{d_1+1,\ldots,d_2\}$. Furthermore, for any $i\in\{1,\ldots,d_1\}$
suppose that
\begin{equation}\label{eq:GammaContractionsConditionHybrid}
\lim_{n\to\infty}\|f_n^{(i)}\s_{r_i}^{\ell_i}
f_n^{(i)}\|=0\quad{and}\quad\lim_{n\to\infty}\|f_n^{(i)}\ts_{q_i/2}^{q_i/2}f_n^{(i)}-c_{q_i}f_n^{(i)}\|=0\quad{with}\quad
c_{q_i}={4\over \left({q_i\over 2}\right)!{q_i\choose q_i/2}^2}
\end{equation}
for all  pairs $(r_i,\ell_i)$ such that either $r_i=q_i$ and $\ell_i=0$, or $r_i\in\{1,\ldots,q_i\}$, $\ell_i\in\{1,\ldots,\min(r_i,q_i-1)\}$
and $r_i$ and $\ell_i$ not equal to $q_i/2$ at the same time. For
$i\in\{d_1+1,\ldots,d_1+d_2\}$ suppose that
\begin{equation}\label{eq:NormalContractionsConditionHybrid}
\lim_{n\to\infty}\|f_n^{(i)}\s_{r_i}^{\ell_i} f_n^{(i)}\|=0
\end{equation}
for all $r_i\in\{1,\ldots,q_i\}$ and
$\ell_i\in\{1,\ldots,\min(r_i,q_i-1)\}$. Then
$\big(I_{q_1}(f_n^{(1)}),\ldots,I_{q_{d_1+d_2}}(f_n^{(d_1+d_2)})\big)$
converges in distribution to
$\big(G_1,\ldots,G_{d_1},N_{d_1+1},\ldots,N_{d_1+d_2}\big)$ as
$n\to\infty$.
\end{theorem}

\begin{example}\label{ex:lilly}{\rm We illustrate Theorem \ref{thm:HybridConvergence} with an example related to the theory of random graphs; the reader is referred to \cite{Pen03} for an introduction to this topic. Note that we will allow the underlying Poisson measure to depend on $n$; see Remark \ref{r:co}. Let $d\geq 1$, and define ${\bf Y}$ and $\eta_n$ as in Section \ref{ss:dejong}. Let $\{r_n : n\geq 1\}$ be a sequence of strictly positive numbers decreasing to zero. For every $n$, we define $D_n := (V_n,E_n)$ to be the random `disk graph' obtained as follows: $V_n = \{Y_i : i=1,\ldots,N(n)\}$ and two vertices $Y_i, Y_j\in V_n$ are connected by an edge if and only if their Euclidean distance is strictly positive and less than $r_n$ (in particular, $D_n$ has no loops). Now let $\Lambda$ be a feasible connected graph (in the sense of \cite{BourPec,LRP2,Pen03}) with $q$ vertices, where $q\neq 2,4$. For every $n$, we define $L_n$ to be the random variable equal to the number of induced subgraphs of $D_n$ that are isomorphic to $\Lambda$, that is, $L_n$ is equal to the number of subsets of the type $Y_{(q)} =\{Y_{i_1},\ldots,Y_{i_q}\}\subset V_n$ such that the restriction of $D_n$ to $Y_{(q)}$ is isomorphic to $\Lambda$. We also set $\widetilde{L}_n = (L_n - \EE[L_n])/{\rm Var}(L_n)^{1/2}$. Now assume that $nr_n^d\to 0$ and $n^q(r_n^d)^{q-1} \to \infty$. According to the discussion contained e.g.\ in \cite[Chapter 3]{Pen03} or \cite[Section 3]{LRP2}, one has that the following four facts are in order: (i) $\EE[L_n] \approx K_0 n^{q}(r_n^d)^{q-1}$ (for some finite constant $K_0>0$), (ii) ${\rm Var}(L_n) \approx K_1 n^{q}(r_n^d)^{q-1}$ (for some finite constant $K_1>0$), (iii) there exists a sequence of multiple integrals of order $q$ with respect to $\hat{\eta}_n = \eta_n - \mu_n$, say $I_q(f_n)$, such that, as $n\to\infty$, $\EE[(\widetilde{L}_n - I_q(f_n))^2]\to 0$, and (iv) the kernels $\{f_n\}$ verify the asymptotic relation (\ref{eq:NormalContractionsConditionHybrid}) (where, for every $n$, the 
contractions and norms have to be considered with respect to the measure $\mu_n$), so that $\widetilde{L}_n$ converges in distribution to a standard Gaussian random variable as $n\to\infty$. Now consider a sequence $\{F'_n\}$ of degenerate $U$-statistics of order 2 as in Theorem \ref{t:gammaU} (for instance, those appearing in Example \ref{ex:rs1}). Since each $F'_n$ is a double Wiener-It\^o integral (with respect to $\eta_n$) verifying condition (\ref{eq:GammaContractionsConditionHybrid}) and $q\neq 4$, we can directly apply Theorem \ref{thm:HybridConvergence} in the case $d_1=d_2 =1$, $q_1=2$ and $q_2 = q$, and conclude that, as $n\to \infty$, the pair $(F'_n , \tilde{L}_n)$ converges in distribution to a vector $(G,N)$ composed of independent random variables such that $G$ has distribution $\cGam_\nu$ and $N$ follows a standard Gaussian random variable.}
\end{example}

We finally show how one can use the results of the present paper to deal with a hybrid Poisson/Gamma convergence (we just consider two-dimensional vectors in order to simplify the discussion, but there is no additional difficulty in considering vectors of higher dimensions).
Let $\nu, \lambda >0$ and let
$(G,P)$ be a vector consisting
of independent random variables such that $G$ has distribution
$\cGam_{\nu}$ and $P$ has a Poisson distribution with mean $\lambda$. We fix an even integer $q\geq 2$, and consider a sequence $\{f_n :n\geq 1\}$ of kernels with $f_n\in L^2_{\rm sym}(\mu_n^{q})$ and such that the technical conditions stated in Section \ref{sec:background}-(VIII) are satisfied. We also consider a sequence $\{H_n :n\geq 1\}$ of random variables such that: (a) each $H_n$ is a functional of the Poisson measure $\eta$, which is in the domain of the Malliavin derivative $D$ and takes values in $\mathbb{Z}_+ = \{0,1,2,\ldots\}$, (b) the numerical sequence
\begin{equation}\label{e:non}
n\mapsto \EE \int_{\mathcal{Z}} (D_zH_n)^2\, \mu_n(\dint z), \qquad n\geq 1,
\end{equation}
is bounded.

\begin{theorem}[{\bf Gamma/Poisson hybrid convergence}]\label{thm:HybridPGConvergence}
Assume that
$\lim\limits_{n\to\infty}q!\|f_n\|^2=2\nu$ and 
\begin{equation}\label{eq:GPH}
\lim_{n\to\infty}\|f_n\s_{r}^{\ell}
f_n\|=0\quad{and}\quad\lim_{n\to\infty}\|f_n\ts_{q/2}^{q/2}f_n-c_{q}f_n\|=0\quad{with}\quad
c_{q}={4\over \left({q\over 2}\right)!{q\choose q/2}^2}
\end{equation}
for all pairs $(r,\ell)$ such that either $r=q$ and $\ell=0$, or $r\in\{1,\ldots,q\}$, $\ell\in\{1,\ldots,\min(r,q-1)\}$
and $r$ and $\ell$ not equal to $q/2$ at the same time. We also assume that, as $n\to\infty$,
\begin{equation}\label{e:hn}
\begin{split}
\EE[H_n]\to\lambda,\qquad &\EE\left| \lambda -\langle DH_n, -DL^{-1}H_n\rangle\right|\to 0,\\ &\mbox{and}\qquad\EE\int_{\mathcal{Z}}\left| D_zH_n (D_zH_n-1) D_zL^{-1}H_n\right|\,\mu_n(\dint z)\to 0.
\end{split}
\end{equation}
Then $\big(I_{q}(f_n),H_n \big)$ converges in distribution to
$\big(G,P\big)$ as $n\to\infty$.
\end{theorem}

\begin{example}\label{ex:lilly2}{\rm We consider the same framework and notation as in Example \ref{ex:lilly}. Here, we take $q\geq 2$ to be a general integer (which can be possibly equal to 2 or 4), whereas $\Lambda$ is a feasible connected graph of order $q$. We stress that, for every $n$, the random variable $L_n$ is a functional of the Poisson measure $\eta_n$ on $\RR^d$, whose control measure is given by $\mu_n(\dint x) = np(x)\dint x$. We put $r_n = n^{-dq/(q-1)}$, in such a way that $n^q(r_n^d)^{q-1} = 1$. According e.g.\ to the analysis contained in \cite[Chapter 3]{Pen03} or \cite[Section 2.4]{BourPec}, one has that the following two facts are in order: (i) there exists a constant $\lambda >0$ such that $\EE[L_n] \approx {\rm Var}(L_n)\approx \lambda n^{q}(r_n^d)^{q-1} = \lambda$, and (ii) the sequence $H_n =L_n$ satisfies (\ref{e:non}), as well as the asymptotic relations (\ref{e:hn})  (here, for every $n$, the Malliavin operators are defined with respect to the ran
 dom measure $\eta_n$ and the 
inner products and integrals are obtained by integrating with respect to $\mu = \mu_n$), so that ${L}_n$ converges in distribution to a Poisson random variable with mean $\lambda$. Considering a sequence $\{F'_n\}$ of degenerate $U$-statistics of order 2 as in Theorem \ref{t:gammaU} (see e.g.\ Example \ref{ex:rs1}), one has that each $F'_n$ is a double integral verifying condition (\ref{eq:GPH}). We can therefore apply Theorem \ref{thm:HybridPGConvergence} and infer that, as $n\to \infty$, the pair $(F'_n , {L}_n)$ converges in distribution to a vector $(G,P)$ composed of independent random variables such that $G$ is distributed according to $\cGam_\nu$ and $P$ has a Poisson distribution with mean $\lambda$.}
\end{example}

\begin{example}\label{ex:lilly3}\rm Let us consider a Poisson measure $\eta_n$ of $k$-dimensional flats in $\RR^d$ with $2k<d$ (where the flats are suitably parameterized to fit into our framework). We assume that the distribution of $\eta_n$ is invariant under rigid motions for each $n$, and that $\eta_n$ has intensity $n\geq 1$. Let us fix a closed convex set $W\subset\RR^d$ with volume one and define the distance $\dist_W(E,F)$ of two $k$-flat $E,F$ as the minimum over the Euclidean distances of $x_E\in E\cap W$ and $x_F\in F\cap W$. By $M_n$ we denote the number of pairs $(E,F)$ of distinct flats of $\eta_n$ such that $\dist_W(E,F)\leq r_n$, where $r_n=n^{-2/(d-2k)}$. According to Theorem 2.1 in \cite{ST12} we know that (i) there exists a constant $0<\lambda<\infty$ (depending on $d,k$ and $W$) such that $\EE[M_n]\approx\VV M_n\approx \lambda$, (ii) the asymptotic relations \eqref{e:hn} are satisfied, and (iii) $M_n$ fulfills the technical condition \eqref{e:non}. Thus, $M_n$ converges in distribution to a Poisson random variable with mean $\lambda$. Let now $\{F'_n\}$ be a sequence of degenerate $U$-statistics of order 2 as in Theorem \ref{t:gammaU}; see Example \ref{ex:rs1}. Then, as in the previous example, each $F'_n$ is a double integral such that condition (\ref{eq:GPH}) is verified. Thus, Theorem \ref{thm:HybridPGConvergence} can be applied to show that the random vector $(F_n',M_n)$ converges in distribution to a random vector $(G,P)$ with independent components such that $G$ has distribution $\cGam_\nu$ and $P$ has a Poisson distribution with mean $\lambda$.
\end{example}

\begin{example}\label{ex:lilly4} \rm Let $W\subset\RR^d$ be a closed convex set with volume one and let ${\bf Y}$ be a sequence of i.i.d.\ points in $W$, which are uniformly distributed and whose random number $N(n)$ follows a Poisson distribution with parameter $n\in\NN$. Any $d+1$ distinct points of $\bf Y$ form a non-degenerate random simplex in $W$. Define $r_n:=n^{-(d+1)}$ and let $V_n$ be the total number of such simplices whose volume does not exceed $r_n$. Then (i) there exists $0<\lambda<\infty$ (depending on $d$ and $W$) such that $\EE[V_n]\approx\VV V_n\approx \lambda$, (ii) the asymptotic relations \eqref{e:hn} are satisfied, and (iii) $V_n$ fulfills the technical condition \eqref{e:non} so that the law of $V_n$ converges, as $n\to\infty$ to a Poisson distribution with mean $\lambda$. This can be seen from Theorem 2.5 in \cite{ST12}. Define the sequence $\{F'_n\}$ of degenerate $U$-statistics as in Example \ref{ex:rs1} or, more generally, as in Theorem \ref{t:gammaU}. Then, following the same 
line of reasoning as as above, Theorem \ref{thm:HybridPGConvergence} can be applied to show that the random vector $(F_n',V_n)$ converges in distribution to a random vector $(G,P)$ with independent components such that $G$ has distribution $\cGam_1$ and $P$ has a Poisson distribution with mean $\lambda$.
\end{example}

\section{Background material}\label{sec:background}

In this section we collect definitions and results that are needed in the statements and proofs of our results. For more details, we refer to the monographs \cite{PeccatiTaqquBook, privaultbook} or to the papers \cite{LastPenrose,NualartVives}.

\paragraph{(I) Poisson measures.}
We shall denote by $\eta$ a Poisson measure with non-atomic and $\sigma$-finite control measure $\mu$ on some Polish space $\sZ$ (which is endowed with the Borel $\sigma$-field $\cZ$). Recall that $\eta$ is a collection $\{\eta(B):B\in\cZ_0\}$ of random variables indexed by the members of $\cZ_0=\{B\in\cZ:\mu(B)<\infty\}$ such that: (a) $\eta(B)$ follows a Poisson distribution with mean $\mu(B)$ for all $B\in\cZ_0$, and (b) whenever $A,B\in\cZ_0$ are disjoint, $\eta(A)$ and $\eta(B)$ are independent random variables. By $P_\eta$ we will denote the distribution of $\eta$ (on the space of $\sigma$-finite counting measures on $\sZ$).

\paragraph{(II) $L^2$-spaces.}
For $q\geq 1$ we denote by $L^2(\mu^q)$ the $L^2$-space $L^2(\sZ^q,\cZ^{\otimes q},\mu^q)$ and by $L_{\rm sym}^2(\mu^q)$ the subspace of $L^2(\mu^q)$ consisting of functions that are $\mu^q$-a.e.\ invariant under permutations of its arguments, so called symmetric functions. Suppressing the dependency on $q$, the scalar product and the norm in $L^2(\mu^q)$ (and $L_{\rm sym}^2(\mu^q)$) are denoted by $\langle\,\cdot\,,\,\cdot\,\rangle$ and $\|\,\cdot\,\|$, respectively. In addition, we let $L^2(P_\eta)$ be the space of square-integrable functionals of $\eta$. To avoid confusion we will use capitals to indicate elements of $L^2(P_\eta)$ and lower cases for elements of $L^2(\mu^q)$ or $L_{\rm sym}^2(\mu^q)$. We finally introduce the space $L^2(\PP,L^2(\mu))=L^2(\Omega\times\sZ,\cF\otimes\cZ,\PP\otimes\mu)$ as the space of jointly measurable mappings $u:\Omega\times\sZ\to\RR$ such that $\EE\int_{\sZ}u(z)^2\,\mu(\dint z)<\infty$ (recall that $(\Omega,\cF,\PP)$ is the underlying probability space).

\paragraph{(III) Multiple stochastic integrals.}
For every integer $q\geq 1$ and every deterministic function $f\in L_{\rm sym}^2(\mu^q)$ let us indicate by $I_q(f)$ the {\it multiple Wiener-It\^o stochastic integral} of order $q$ of $f$ with respect to the compensated Poisson measure $\eta-\mu$. For general $f\in L^2(\mu^q)$ we put $I_q(f):=I_q(\widetilde{f})$, where $\widetilde{f}(x_1,\ldots,x_q)=(q!)^{-1}\sum_{\pi}f(x_{\pi(1)},\ldots,x_{\pi(q)})$ is the canonical symmetrization of $f$ and the sum in its definition runs over all $q!$ permutations $\pi$ of $\{1,\ldots,q\}$. The multiple stochastic integrals satisfy the following properties: $$\EE[I_{q_i}(f_i)]=0\quad(i=1,2), \quad{\rm and}\quad\EE\big[I_{q_1}(f_1)I_{q_2}(f_2)\big]=q_1!\langle f_1,f_2\rangle\,{\bf 1}(q_1=q_2)$$ for any $q_1,q_2\in\{1,2,\ldots\}$, $f_1\in L_{\rm sym}^2(\mu^{q_1})$ and $f_2\in L_{\rm sym}^2(\mu^{q_2})$.

\paragraph{(IV) Chaotic representation property.}
The $q$-th Wiener chaos $W_q$ associated with the Poisson measure $\eta$ is the Hilbert space $W_q=\{I_q(f):f\in L_{\rm sym}^2(\mu^q)\}$. In addition, we put $W_0:=\RR$. It is a crucial property of $\eta$ that $L^2(P_\eta)$ can be written as a direct sum of Wiener chaoses, i.e.\ $L^2(P_\eta)=\bigoplus\limits_{q=0}^\infty W_q$. As a consequence, every $F\in L^2(P_\eta)$ admits a (unique) chaotic decomposition in the sense that
\begin{equation}\label{eq:ChaosDecomposition}
F=\EE[F]+\sum_{q=1}^\infty I_q(f_q)
\end{equation}
with suitable functions $f_q\in L_{\rm sym}^2(\mu^q)$ and where the series converges in $L^2(P_\eta)$.

\paragraph{(V) The Malliavin operators $D$, $L^{-1}$ and $\delta$.}
The domain ${\rm dom}\,D$ of the {\it derivative operator} $D$ is the set of all $F\in L^2(P_\eta)$ admitting a chaos decomposition \eqref{eq:ChaosDecomposition} such that $\sum\limits_{q=1}^\infty q\,q!\|f_q\|^2<\infty$. For such $F$ the random function $\sZ\ni z\mapsto D_zF$ is defined by 
\begin{equation}\label{e:df}
D_zF=\sum_{q=1}^\infty qI_{q-1}\big(f_q(z,\,\cdot\,)\big),
\end{equation} 
where $f_q(z,\,\cdot\,)$ is the function $f_q$ with one of its argument fixed to be $z$. Notice that $DF\in L^2(\PP,L^2(\mu))$. The derivative operator can be also characterized as an ``add-one cost operator'', as follows; see \cite{LastPenrose,NualartVives} for proofs of this fact. For $F\in L^2(P_\eta)$ and $z\in\sZ$, let $F_z(\eta)$ be the random variable $F(\eta+\d_z)$. Then, for $F\in{\rm dom}\,D$ and $\mu$-almost every $z\in\sZ$, we have the identity $D_zF=F_z-F$, a.s.-$\PP$. Throughout the text, we also implicitly use the following converse statement (the proof is an elementary consequence of the main findings of \cite{LastPenrose}, and is included for the sake of completeness).

\begin{lemma}\label{l:rip} Let $F\in L^2(P_\eta)$ be such that $\EE\int_{\sZ} (F_z - F)^2\, \mu(\dint z)<\infty$. Then, $F\in {\rm dom}\, D$.
\end{lemma}
\begin{proof} For every $z\in \mathcal{Z}$, define the `trajectorial' difference operator $D'_zF(\eta) = F_z(\eta) - F(\eta)$. According to \cite[Theorem 1.3]{LastPenrose}, the square-integrable random variable $F$ admits a chaotic decomposition of the type (\ref{eq:ChaosDecomposition}), with
$$
f_q(z_1,\ldots,z_q) = \frac{1}{q!}\,\EE[D'_{z_1}\ldots D'_{z_q} F], \quad q=1,2,\ldots
$$
(in particular, the deterministic function on the right-hand side of the previous equation is a well-defined element of $L^2_{\rm sym}(\mu^q)$ for every $F\in L^2(P_\eta)$ and every $q\geq 1$). In view of the assumptions, there exists a measurable set $\mathcal{Z}'$ such that $\mu(\mathcal{Z}\backslash \mathcal{Z}') = 0$ and $ \EE[(D'_zF)^2]= \EE[(F_z - F)^2] <\infty$ for every $z\in \mathcal{Z}'$. It follows that the statement is proved once we show that, for every $z\in \mathcal{Z}'$, the chaotic decomposition of $D_z'F$ coincides with the right-hand side of (\ref{e:df}). Again by virtue of \cite[Theorem 1.3]{LastPenrose}, one has that the $q$th integrand in the chaotic decomposition of $D'_zF$ is given by the mapping
$$
(z_1,\ldots,z_q) \mapsto \frac{1}{q!}\,\EE[D'_{z_1}\ldots D'_{z_q} D'_zF] = (q+1) f_{q+1}(z,z_1,\ldots,z_q),
$$
which yields the desired conclusion.
\end{proof}

\noindent For any $F\in L^2(P_\eta)$ with chaotic decomposition \eqref{eq:ChaosDecomposition} satisfying $\EE[F]=0$ we put $$L^{-1}F=-\sum_{q=1}^\infty q^{-1}I_q(f_q).$$ The operator $L^{-1}$ is 
the so-called {\it pseudo-inverse of the Ornstein-Uhlenbeck generator}. Finally, we observe
that, due to the chaotic representation property of $\hat{\eta}$,
every random function $u \in L^2(\mathbb{P},L^2(\mu)) $ admits a
(unique) representation of the type
\begin{equation}\label{processes}
u_z = \sum_{q=0}^\infty I_q(f_q(z,\,\cdot\,)),\qquad z\in\sZ,
\end{equation}
where, for every $z$, the kernel $f_q(z,\,\cdot\,)$ is an element of
$L^2_{\rm sym}(\mu^q)$. The domain of the {\it divergence operator},
denoted by ${\rm dom}\, \delta$, is defined as the collections of
those $u\in L^2(\mathbb{P},L^2(\mu))$ such that the chaotic
expansion (\ref{processes}) verifies the condition
\begin{equation*}
\sum_{q=0}^\infty(q+1)! \|f_q\|^2<\infty.
\end{equation*}
If $u \in {\rm dom}\, \delta$, then the random variable $\delta (u)$
is defined as
\begin{equation*}
\delta(u) = \sum_{q=0}^\infty I_{q+1}(\tilde{f}_q),
\end{equation*}
where $\tilde{f}_q$ stands for the canonical symmetrization of $f_q$
(as a function in $q+1$ variables). The following classic result, proved e.g.\ in
\cite{NualartVives}, yields a characterization of $\delta$ as
the adjoint of the derivative $D$.
\begin{lemma}[{\bf Integration by parts formula}]\label{L : SkorohodAsAdjoint} For every $G \in {\rm
dom}\,D$ and every $u \in {\rm dom}\, \delta$, one has that
\begin{equation}\label{dualDdelta}
\mathbb{E} [G \delta(u)] = \mathbb{E}[\langle DG , u   \rangle],
\end{equation}
where, more explicitly,
$$
\langle DG, u  \rangle = \int_{\mathcal Z} D_z G \times u(z)\,\mu(\dint z).
$$
\end{lemma}

\paragraph{(VI) Contractions.}
Let $f\in L_{\rm sym}^2(\mu^q)$ for some integer $q\geq 1$ and $r\in\{0,\ldots,p\}$, $\ell\in\{1,\ldots,r\}$. The contraction kernel $f\s_r^\ell f$ on $\sZ^{2q-r-\ell}$ acts on the tensor product $f\otimes f$ first by identifying $r$ variables and then integrating out $\ell$ among them. More formally,
\begin{equation*}
\begin{split}
f\s_r^\ell f(\g_1,\ldots,\g_{r-\ell},t_1,\ldots,t_{q-r},&s_1,\ldots,s_{q-r}) = \int_{\sZ^\ell}f(z_1,\ldots,z_\ell,\g_1,\ldots,\g_{r-\ell},t_1,\ldots,t_{q-r})\\
&\times\,f(z_1,\ldots,z_\ell,\g_1,\ldots,g_{r-\ell},s_1,\ldots,s_{q-r})\,\mu^\ell\big(\dint(z_1,\ldots,z_\ell)\big).
\end{split}
\end{equation*}
In addition, we put
\begin{equation*}
f\s_r^0f(\g_1,\ldots,\g_r,t_1,\ldots,t_{q-r},s_1,\ldots,s_{q-r})=f(\g_1,\ldots,\g_r,t_1,\ldots,t_{q-r})f(\g_1,\ldots,\g_r,s_1,\ldots,s_{q-r}).
\end{equation*}
Besides the contraction $f\s_r^\ell f$, we will also deal with its canonical symmetrization $f\ts_r^\ell f$, which is defined as $$(f\ts_r^\ell f)(x_1,\ldots,x_{2q-r-\ell})={1\over (2q-r-\ell)!}\sum_\pi(f\s_r^\ell f)(x_{\pi(1)},\ldots,x_{\pi(2q-r-\ell)}),$$ where the sum runs over all $(2q-r-\ell)!$ permutations of the set $\{1,\ldots,2q-r-\ell\}$.

\paragraph{(VII) Product formula.}
Let $q_1,q_2\geq 1$ be integers, $f_1\in L_{\rm sym}^2(\mu^{q_1})$ and $f_2\in L_{\rm sym}^2(\mu^{q_2})$ be as in the previous paragraph. In terms of the contractions of $f_1$ and $f_2$ one can express the product of $I_{q_1}(f_1)$ and $I_{q_2}(f_2)$ as follows:
\begin{equation}\label{eq:ProductFormula}
I_{q_1}(f_1)I_{q_2}(f_2)=\sum_{r=0}^{\min(q_1,q_2)}r!{q_1\choose r}{q_2\choose r}\sum_{\ell=0}^r{r\choose\ell}I_{q_1+q_2-r-\ell}(f_1\ts_r^\ell f_2)\,;
\end{equation}
see \cite[Proposition 6.5.1]{PeccatiTaqquBook}. In the particular case $q_1=q_2=:q$ and $f_1=f_2$, we may define $G_0^qf:=q!\|f\|^2$ and
\begin{equation*}\label{eq:DefGpq}
G_p^qf:=\sum_{r=0}^q\sum_{\ell=0}^r{\bf 1}(2q-r-\ell=p)\,r!{q\choose r}^2{r\choose\ell}f\ts_r^\ell f
\end{equation*}
for $p\in\{1,\ldots,2q\}$, which allows us to re-write \eqref{eq:ProductFormula} in the more compact form
\begin{equation}\label{eq:ProductFormulaWithG}
I_q^2(f)=\sum_{p=0}^{2q}I_p(G_p^qf).
\end{equation}

\paragraph{(VIII) Technical assumptions.}
Whenever we deal with a multiple stochastic integral or a sequence $I_q(f_n)$ of such integrals with $f_n\in L_{\rm sym}^2(\mu_n^q)$ we will (implicitly) assume that the following technical conditions are satisfied:
\begin{itemize}
 \item[i)] for any $r\in\{1,\ldots,q\}$, the contraction $f_n\s_r^{q-r}f_n$ is an element of $L^2(\mu_n^r)$;
 \item[ii)] for any $r\in\{1,\ldots,q\}$, $\ell\in\{1,\ldots,r\}$ and $(z_1,\ldots,z_{2q-r-\ell})\in\sZ^{2q-r-\ell}$ we have that $(|f_n|\s_r^\ell|f_n|)(z_1,\ldots,z_{2q-r-\ell})$ is well defined and finite;
 \item[iii)] for any $k\in\{0,\ldots,2(q-1)\}$ and any $r$ and $\ell$ satisfying $k=2(q-1)-r-\ell$ we have that $$\int_{\sZ}\sqrt{\int_{\sZ}\big(f_n(z,\,\cdot\,)\s_r^\ell f_n(z,\,\cdot\,)\big)^2\,\dint\mu_n^{k}}\,\mu_n(\dint z)<\infty.$$
\end{itemize}

We remark that (iii) is automatically satisfied if the control measure $\mu$ is finite (which is the case in our Examples \ref{ex:lilly} and \ref{ex:lilly2}-\ref{ex:lilly4}). Intuitively, conditions (i)-(iii) ensure that every manipulation involving contraction kernels performed below is justified and is in fact valid. For the detailed role of these conditions and their implications we refer to \cite{LRP} or \cite{PSTU2010}.

\section{Proofs of the results}\label{sec:proofs}

\subsection{Proof of Theorem \ref{thm:SteinBound}}\label{sec:proofSteinBound}

Before entering the details of the proof of Theorem \ref{thm:SteinBound} we recall some facts related to Stein's method for the Gamma distribution established by Luk in \cite{Luk}; see also Pickett \cite{Pickett} for refinements in the case of an integer-valued parameter $\nu$. We start by considering the \textit{second-order} Stein equation 
\begin{equation}\label{e:luk}
h(x-\nu)-\EE[h\big(G^*(\nu)\big)]=2xg''(x)-(x-\nu)g'(x),\qquad x>0,
\end{equation}
where $G^*(\nu)=G(\nu)+\nu$, with $G(\nu)$ distributed according to $\cGam_\nu$, and $h\in \mathcal{H}^3$. It is shown in \cite[Theorem 1]{Luk} that (\ref{e:luk}) admits a solution $V_h$ such that $\|V_h^{(j)}\|_\infty\leq {2\over j}\|h^{(j)}\|_\infty$ for $j=1,2,3$. Note that the assumption $h\in \mathcal{H}^3$ automatically yields that $h$ has sub-exponential growth, so that \cite[Theorem 1]{Luk} can directly be applied.

Now, we turn to the {\it first-order} Stein operator $T$ for $\cGam_\nu$, which acts on differentiable functions $f:\RR\to\RR$. It is given by $$Tf(x)=2(x+\nu)_+f'(x)-xf(x),\qquad x\in\RR.$$ The associated first-order Stein equation is
\begin{equation*}\label{eq:SteinEquation}
h(x)-\EE[h\big(G(\nu)\big)]=Tf(x),\qquad x\in\RR, 
\end{equation*}
where $h\in \mathcal{H}^3$. For such an $h$, a solution $U_h$ of the Stein equation -- in what follows, sometimes called a {\it Stein solution} -- is provided by
\begin{equation}\label{eq:SteinSolution}
U_h(x)=\begin{cases}-{1\over x}\big(h(x)-\EE[h(G(\nu))\big)] &: \,\,x\leq -\nu\\ V_h'(x+\nu) &:\,\, x>-\nu.\end{cases}
\end{equation}
Recall that the probability density of $G(\nu)$ is given by $$g_\nu(x) = {2^{-\nu/2}\over{\Gamma(\nu/2)}}{(x+\nu)^{\nu/2 - 1}}e^{-(x+\nu)/2}\,{\bf 1}_{\{x>-\nu\}}.$$ Since for our choice of the test function $h$ the mapping $x\mapsto U_h(x) =V_h'(x+\nu)$ is bounded on $(\nu, \infty)$, we can use \cite[Lemma 4]{Stein} to deduce that, necessarily,
$$
U_h(x) =\frac{1}{2(x+\nu)g_\nu(x)} \int_{-\nu}^x\big(h(y)-\EE[h(G(\nu))]\big)\,g_\nu(y)\,\dint y,\qquad x>-\nu,
$$
yielding that $x\mapsto U_h(x)$ is continuous on $\RR$, as deduced from a simple application of de l'H\^opital's rule at $x=-\nu$. Also, $U_h$ is twice differentiable on $\RR\setminus\{\nu\}$ and satisfies the estimates $\|U_h\|_\infty\leq\max(2,2/\nu)=:c_0$, $\|U_h'\|_\infty\leq\max(1,1/\nu+1/\nu^2)=c_1$ and $\|U_h''\|_\infty\leq\max(2/3,2/(3\nu)-3/\nu^2+4/\nu^3)=c_2$ (here, $c_1$ and $c_2$ are the constants from Theorem \ref{thm:SteinBound}). We stress that, albeit $U_h(x)$ is continuous on $\RR$, such a function is in general not differentiable at $x=-\nu$ (it is however right- and left-differentiable at such a point). We remark that the quantities $2$, $1$ and $2/3$ appearing in the constants $c_0,c_1$ and $c_2$ come from smoothness estimates for (\ref{eq:SteinSolution}) on the interval $(-\nu,\infty)$, whereas the presence of the constants $2/\nu$, $1/\nu+1/\nu^2$ and $2/(3\nu)-3/\nu^2+4/\nu^3$ is explained by elementary estimates of \eqref{eq:SteinSolution} on the interval $(-\infty,-\nu]$. 

Let now $\cF^2$ be the space of continuous functions $f$ on $\RR$, which are twice differentiable on $\RR\setminus\{\nu\}$ and satisfy $$\|f\|_\infty\leq c_0,\quad\|f'\|_\infty\leq c_1\quad{\rm and}\quad\|f''\|_\infty\leq c_2.$$ In the light of the previous discussion, we conclude that
\begin{equation}\label{eq:d2SteinBoundBeginnging}
d_3(F,\cGam_\nu) \leq \sup_{f\in\cF^2}\left|\EE\big[2(F+\nu)_+f'(F)-Ff(F)\big]\right|,
\end{equation}
where $F$ and $\nu$ are as in the statement of Theorem \ref{thm:SteinBound} and where here and below $f'(-\nu)$ stands for the left-sided derivative of $f$ at $-\nu$, i.e.\ $f'(-\nu)=\lim\limits_{x\nearrow -\nu}f'(x)$. We also refer the reader to Lemma 1.3 in \cite{NourdinPeccatiWienerChaos} and the references cited therein. The estimate \eqref{eq:d2SteinBoundBeginnging} is the starting point of the proof of Theorem \ref{thm:SteinBound}.

\begin{proof}[Proof of Theorem \ref{thm:SteinBound}]
We have to show that the right-hand side of \eqref{eq:d2SteinBoundBeginnging} is bounded from above by the right-hand side of \eqref{eq:d2SteinBoundGeneral}. This is done by borrowing some ideas from \cite{SchulteKolm}. To start with, consider $f\in\cF^2$, write $F(\eta)$ instead of $F$ to emphasize the dependency of $F$ on $\eta$ and fix $z\in\sZ$. Because of the non-differentiability of the Stein-solution at $-\nu$, we will have to distinguish the three cases a) $F(\eta)\leq-\nu$, $F(\eta+\d_z)\leq -\nu$ or $F(\eta)>-\nu$, $F(\eta+\d_z)>-\nu$, b) $F(\eta)\leq-\nu<F(\eta+\d_z)$ and c) $F(\eta+\d_z)\leq-\nu<F(\eta)$. For a) we use a Taylor expansion to see that
\begin{equation*}
\begin{split}
D_zf\big(F(\eta)\big) &= f\big(F(\eta+\d_z)\big)-f\big(F(\eta)\big)\\
&=  f'\big(F(\eta)\big)\big(F(\eta+\d_z)-F(\eta)\big)+R\big(F(\eta+\d_z)-F(\eta)\big)\\
&= f'\big(F(\eta)\big)D_zF(\eta)+R_a\big(D_zF(\eta)\big),
\end{split}
\end{equation*}
where the reminder $R_a$ is such that $|R_a(x)|\leq {1\over 2}\|f''\|_\infty\, x^2={1\over 2}c_2x^2$; recall that $f$ is differentiable on $\RR\setminus\{-\nu\}$, as well as right- and left-differentiable at $x=-\nu$. For case b) we also use a Taylor expansion to see that
\begin{equation*}
\begin{split}
D_zf\big(F(\eta)\big) &= f\big(F(\eta+\d_z)\big)-f\big(F(\eta)\big)=f\big(F(\eta+\d_z)\big)-f(-\nu)+f(-\nu)-f\big(F(\eta)\big)\\
&=f'(-\nu+)\big(F(\eta+\d_z)+\nu\big)+{1\over 2}f''\big(\widetilde{F}(\eta)\big)\big(F(\eta+\d_z)+\nu\big)^2\\
&\qquad +\,f'\big(F(\eta)\big)\big(-\nu-F(\eta)\big)+{1\over 2}f''\big(\hat F(\eta)\big)\big(-\nu-F(\eta)\big)^2\\
&=f'\big(F(\eta)\big)D_zF(\eta)-f'\big(F(\eta)\big)\big(F(\eta+\d_z)+\nu\big)+f'(-\nu+)\big(F(\eta+\d_z)+\nu\big)\\
&\qquad +\,{1\over 2}f''\big(\widetilde{F}(\eta)\big)\big(F(\eta+\d_z)+\nu\big)^2+{1\over 2}f''\big(\hat F(\eta)\big)\big(F(\eta)+\nu\big)^2\\
&=:f'\big(F(\eta)\big)D_zF(\eta)+R_b\big(F(\eta),z,\nu\big)
\end{split}
\end{equation*}
with some $\widetilde{F}\in \big(-\nu,F(\eta+\d_z)\big) ,\hat F\in\big(F(\eta),-\nu\big)$ and where $f'(-\nu+)$ stands for the right-sided derivative of $f$ at $-\nu$. Similarly, in case c) we find that
\begin{equation*}
\begin{split}
D_zf\big(F(\eta)\big) &= f\big(F(\eta+\d_z)\big)-f\big(F(\eta)\big)=f\big(F(\eta+\d_z)\big)-f(-\nu)+f(-\nu)-f\big(F(\eta)\big)\\
&=f'(-\nu-)\big(F(\eta+\d_z)+\nu\big)+{1\over 2}f''\big(\widetilde{F}(\eta)\big)\big(F(\eta+\d_z)+\nu\big)^2\\
&\qquad +\,f'\big(F(\eta)\big)\big(-\nu-F(\eta)\big)+{1\over 2}f''\big(\hat F(\eta)\big)\big(-\nu-F(\eta)\big)^2\\
&=f'\big(F(\eta)\big)D_zF(\eta)-f'\big(F(\eta)\big)\big(F(\eta+\d_z)+\nu\big)+f'(-\nu-)\big(F(\eta+\d_z)+\nu\big)\\
&\qquad +\,{1\over 2}f''\big(\widetilde{F}(\eta)\big)\big(F(\eta+\d_z)+\nu\big)^2+{1\over 2}f''\big(\hat F(\eta)\big)\big(F(\eta)+\nu\big)^2\\
&=:f'\big(F(\eta)\big)D_zF(\eta)+R_c\big(F(\eta),z,\nu\big)
\end{split}
\end{equation*}
again with some $\widetilde{F}\in \big(F(\eta+\d_z),-\nu\big),\hat F\in\big(-\nu,F(\eta)\big)$ and where $f'(-\nu-)$ stands for the left-sided derivative of $f$ at $\nu$. Summarizing, we conclude that 
\begin{equation}\label{eq:TaylorExpansion}
D_zf\big(F(\eta)\big)=f'\big(F(\eta)\big)D_zF(\eta)+R\big(F(\eta),z,\nu\big)
\end{equation}
(recall that $f'(-\nu)=f'(-\nu-)$ by convention), where the global reminder term $R\big(F(\eta),z,\nu\big)$ is given by
\begin{equation*}
\begin{split}
R\big(F(\eta),z,\nu\big) =& \,R_a\big(F(\eta)\big)\big({\bf 1}_{\{F(\eta),F(\eta+\d_z)>-\nu\}}+{\bf 1}_{\{F(\eta),F(\eta+\d_z)\leq-\nu\}}\big)\\
&+\,R_b\big(F(\eta),z,\nu\big)\,{\bf 1}_{\{F(\eta)\leq-\nu<F(\eta+\d_z)\}}+\,R_c\big(F(\eta),z,\nu\big)\,{\bf 1}_{\{F(\eta+\d_z)\leq-\nu<F(\eta)\}}.
\end{split}
\end{equation*}
We have seen that $R_a$ has the property that $|R_a(x)|\leq {1\over 2}\|f''\|_\infty\, x^2={1\over 2}c_2x^2$. For $R_b$ and $R_c$ we notice that in these cases $\big|F(\eta+\d_z)+\nu\big|\leq\big|D_zF(\eta)\big|$ and $\big|F(\eta)+\nu\big|\leq\big|D_zF(\eta)\big|$, which together with the properties of $f\in\cF^2$ leads to the bound
\begin{equation*}
\begin{split}
\big|R\big(F(\eta),z,\nu\big)\big| 
&\leq {1\over 2}c_2\big|D_zF(\eta)\big|^2\big({\bf 1}_{\{F(\eta),F(\eta+\d_z)>-\nu\}}+{\bf 1}_{\{F(\eta),F(\eta+\d_z)\leq-\nu\}}\big)\\
&\qquad +\,\big(2c_1\big|D_zF(\eta)\big|+c_2\big|D_zF(\eta)\big|^2\big)\,{\bf 1}_{\{F(\eta)\leq-\nu<F(\eta+\d_z)\}}\\
&\qquad +\,\big(2c_1\big|D_zF(\eta)\big|+c_2\big|D_zF(\eta)\big|^2\big)\,{\bf 1}_{\{F(\eta +\d_z)\leq-\nu<F(\eta)\}}\\
&\leq c_2\big|D_zF(\eta)\big|^2+2c_1\big|D_zF(\eta)\big|\big({\bf 1}_{\{F(\eta)\leq-\nu<F(\eta+\d_z)\}}+{\bf 1}_{\{F(\eta+\d_z)\leq-\nu<F(\eta)\}}\big)\\
&=c_2\big|D_zF(\eta)\big|^2+2c_1\big(D_z{\bf 1}_{\{F(\eta)>-\nu\}}\big)\big(D_zF(\eta)\big).
\end{split}
\end{equation*}
Using now the integration by parts formula from Malliavin calculus, \eqref{dualDdelta} in Lemma \ref{L : SkorohodAsAdjoint}, and simplifying the resulting expression we find 
$$\EE\big[Ff(F)\big]=\EE\big[LL^{-1}Ff(F)\big]=\EE\big[-\delta(DL^{-1}F)f(F)\big]=\EE\big[\langle Df(F),-DL^{-1}F\rangle\big],$$ which in view of \eqref{eq:TaylorExpansion} leads to $$\EE\big[\langle Df(F),-DL^{-1}F\rangle\big]=\EE\big[f'(F)\langle DF,-DL^{-1}F\rangle\big]+\EE\big[\langle R(F,z,\nu),-DL^{-1}F\rangle \big].$$ Consequently, because of the above estimate on $\big|R\big(F(\eta),z,\nu\big)\big|$,
\begin{equation*}
\begin{split}
&\left|\EE\big[2(F+\nu)_+f'(F)-Ff(F)\big]\right|\\
&\leq \left|\EE\big[f'(F)(2(F+\nu)_+-\langle DF,-DL^{-1}F\rangle)\big]\right|+\left|\EE\big[\langle R(F,z,\nu),-DL^{-1}F\rangle \big]\right|\\
&\leq c_1\EE\left|2(F+\nu)_+-\langle DF,-DL^{-1}F\rangle\right|+c_2\int_{\sZ}\EE\big[|D_zF|^2|DL^{-1}F|\big]\,\mu(\dint z)\\
&\hspace{3cm} +\,2c_1\int_{\sZ}\EE\big[\big(D_z{\bf 1}_{\{F>-\nu\}}\big)\big(D_zF\big)|D_zL^{-1}F|\big]\,\mu(\dint z).
\end{split}
\end{equation*}
This shows the first inequality \eqref{eq:d2SteinBoundGeneral} in Theorem \ref{thm:SteinBound}. The second estimate \eqref{eq:d2SteinBoundWithoutPlus} follows from \eqref{eq:d2SteinBoundGeneral} and the assumption that $\EE\big[\langle DF,-DL^{-1}F\rangle|F\big]\geq 0$. This proves Theorem \ref{thm:SteinBound}.
\end{proof}

\subsection{Proof of Proposition \ref{l:newbound}}

We start by observing that the function $x\mapsto \Phi(x) := x|x| = {\rm sign}(x)\, x^2$, $x\in \RR$, is such that, for every $a,b\in \RR$, $\Phi(b) = \Phi(a) + 2|a| (b-a) +R(a,b)$, where $|R(a,b)|\leq (b-a)^2$. It follows that
\begin{equation}\label{e:gns}
(\Phi(b) - \Phi(a))^2 \leq 8a^2(b-a)^2 +2(b-a)^4.
\end{equation}
Since $\mu$ is finite, $$\EE\int_{\sZ}(D_z{\bf 1}_{\{F>-\nu\}})^2\,\mu(\dint z)\leq\EE\int_{\sZ}({\bf 1}_{\{F+D_zF>-\nu\}}-{\bf 1}_{\{F>-\nu\}})^2\,\mu(\dint z)\leq \mu(\sZ)<\infty,$$ which implies that ${\bf 1}_{\{F>-\nu\}} \in {\rm dom}\, D$; see Lemma \ref{l:rip} and compare with Remark \ref{rem:IndicatorInDomD} (i). Moreover, our assumptions imply that $DF|DF| = \Phi(DF)\in {\rm dom}\, \delta$. We can now apply the integration by parts formula (\ref{dualDdelta}), together with the relation $L^{-1}F = -q^{-1}F$, to deduce that
\begin{eqnarray*} {q}\times A_3(F) &=&  \EE\int_{\mathcal Z} (D_z{\bf 1}_{\{F>-\nu\}})\Phi(D_zF)\,\mu(\dint z)  \\
&=& \EE[{\bf 1}_{\{F>-\nu\}}\delta(\Phi(DF))]\\
& \leq & \big[\EE[\delta(\Phi(DF))^2]\big]^{1/2}.
\end{eqnarray*}
Again in view of our assumptions, the Skorohod isometry implied by \cite[Proposition 6.5.4]{privaultbook} is verified, and we deduce that
\begin{eqnarray*}
\EE[\delta(\Phi(DF))^2] &\leq & \EE\int_{\mathcal Z} \Phi(D_zF)^2\,\mu(\dint z) + \EE\int_{\mathcal Z}\int_{\mathcal Z} [D_{z_2}\Phi(D_{z_1}F)]^2\,\mu(\dint z_1)\mu(\dint z_2)\\
&=& \EE\int_{\mathcal Z} (D_zF)^4\,\mu(\dint z) + \EE\int_{\mathcal Z}\int_{\mathcal Z} [D_{z_2}\Phi(D_{z_1}F)]^2\,\mu(\dint z_1)\mu(\dint z_2).
\end{eqnarray*}
Since $D_{z_2}\Phi(D_{z_1}F) = \Phi(D_{z_1}F+ D_{z_2}D_{z_1}F) -\Phi(D_{z_1}F)$, we can now apply (\ref{e:gns}) with $a = D_{z_1}F$ and $b =  D_{z_1}F + D_{z_2}D_{z_1}F$ to infer the upper bound 
$$
[D_{z_2}\Phi(D_{z_1}F)]^2 \leq 8(D_{z_1}F)^2(D_{z_2}D_{z_1} F)^2 + 2 (D_{z_2}D_{z_1}F)^4,
$$
and the conclusion follows immediately.\hfill $\Box$

\subsection{Proof of Theorem \ref{thm:GammaContractions}}

Let $F_n=I_q(f_n)$ be as in the statement of Theorem \ref{thm:GammaContractions}. Then $\langle DF_n,-DL^{-1}F_n\rangle={1\over q}\|DI_q(f_n)\|^2$ and $\EE\big[\langle DF_n,-DL^{-1}F_n\rangle|F_n\big]\geq 0$. Thus, we need to prove that for such $F_n$ the right-hand side of \eqref{eq:d2SteinBoundWithoutPlus} converges to zero as $n\to\infty$. We do this by showing that the three terms $A_1'(F_n)$, $A_3(F_n)$ and $A_4(F_n)$ (see (\ref{e:tgv})) all converge to zero as $n\to \infty$; the computations performed below will also implicitly provide the upper bound (\ref{e:estimate}). It is important to note that our analysis of the terms $A_1'(F_n)$ and $A_4(F_n)$ does not make use of the fact that $\mu_n(\mathcal{Z})<\infty$. It is convenient to start with the reminder term $A_4(F_n)$.

\begin{lemma}\label{lem:GammaStep2}
Under the conditions of Theorem \ref{thm:GammaContractions}, it holds that $A_4(I_q(f_n)) \to 0$, as $n\to\infty$.
\end{lemma}
\begin{proof}
First observe that in our case $$A_4(I_q(f_n)) =\sqrt{ \int\limits_{\sZ}\EE[|D_zI_q(f_n)|^4]\,\mu_n(\dint z)}.$$ We can now use \cite[formulae (4.17) and (4.18)]{PSTU2010} to deduce that
\begin{equation}
\begin{split}\label{e:fire}
A_4(I_q(f_n)) \leq q^2&\sum_{r=1}^q\sum_{\ell=0}^{r-1}{\bf 1}(1\leq r+\ell\leq 2q-1)\\ 
&\times\,\big((r+\ell-1)!\big)^{1/2}(q-\ell-1)!{q-1\choose q-1-\ell}^2{q-1-\ell\choose q-r}\|f_n\s_r^\ell f_n\|.
\end{split}
\end{equation}
Since this estimate does not involve the middle contraction $f_n\s_{q/2}^{q/2}f_n$, the conclusion follows immediately.
\end{proof}

Now we study the convergence of the sequence $A_1'(F_n)$.

\begin{lemma}\label{lem:GammaStep1}
Under the conditions of Theorem \ref{thm:GammaContractions} we have $ A_1'(I_q(f_n)) \to 0$, as $n\to\infty$.
\end{lemma}
\begin{proof}
One must prove that $\EE\big[\|DI_q(f_n)\|^2-2qI_q(f_n)-2q\nu\big]^2\to 0$. Expanding the square and using the fact that $\EE[I_q(f_n)]=0$ we have to show that
\begin{equation}\label{eq:ProofLemma2Condition}
\EE[\|DI_q(f_n)\|^4]-4q\EE[I_q(f_n)\|DI_q(f_n)\|^2]+4q^2\EE[I_q^2(f_n)]-4q\nu\EE[\|DI_q(f_n)\|^2]+4q^2\nu^2\to 0
\end{equation}
as $n\to\infty$. Firstly, $\EE[I_q^2(f_n)]=q!\|f_n\|^2\rightarrow 2\nu$. The definition of $DI_q(f_n)$ and formula \eqref{eq:ProductFormulaWithG} imply that
\begin{equation}\label{eq:D2Proof}
\|DI_q(f_n)\|^2=q\,q!\|f_n\|^2+q^2\sum_{p=1}^{2(q-1)}\int_{\sZ}I_p\big(G_p^{q-1}f_n(z,\,\cdot\,)\big)\,\mu_n(\dint z)
\end{equation}
so that $\EE[\|DI_q(f_n)\|^2]=q\,q!\|f_n\|^2$, which asymptotically behaves like $2q\nu$. Using integration by parts \eqref{dualDdelta} together with the relation $DF^2 = 2FDF +(DF)^2$ applied to $F = I_q(f)$, we infer that 
$$
\EE[I_q(f_n) \|DI_q(f_n)\|^2] = \frac{q}{2} \EE[I_q^3(f_n)] - \frac12\,\EE\int_\sZ D_zI_q^3(f_n)\,\mu_n(\dint z).
$$ 
Now, in view of the estimate \eqref{e:fire}, the second summand on the right-hand side of the previous equation converges to zero as $n\to \infty$, and consequently $\EE\big[I_q(f_n)\|DI_q(f_n)\|^2\big]$ behaves asymptotically as ${q\over 2}\EE[I_q^3(f_n)]$. Using \eqref{eq:ProductFormula} and the orthogonality of chaoses we obtain
\begin{equation*}\label{eq:3rdMoment}
\begin{split}
\EE[I_q^3(f_n)] &= \sum_{p=0}^qp!{q\choose p}^2\sum_{\ell=0}^p{p\choose\ell}\,\EE\big[I_{2q-p-\ell}(f_n)I_q(f_n)\big]\\ 
&= \sum_{p=q/2}^qp!{q\choose p}^2{p\choose q-p}q!\langle f_n\ts_p^{q-p}f_n,f_n\rangle,
\end{split}
\end{equation*}
so that $\EE\big[I_q(f_n)\|DI_q(f_n)\|^2\big]$ has the same limit as $${q\over 2}\sum_{p=q/2}^q p!{q\choose p}^2{p\choose q-p}q!\langle f_n\ts_p^{q-p}f_n,f_n\rangle.$$
Moreover, one can show that

\begin{equation}\label{eq:D4Proof}
\EE[\|DI_q(f_n)\|^4]=q^2(q!\|f_n\|^2)^2+q^4\sum_{p=1}^{2(q-1)}p!\|{\hat{ G}_p^qf_n}\|^2,
\end{equation}
where $\hat G_p^qf_n$ with $p\in\{1,\ldots,2(q-1)\}$ is defined by $$\hat G_p^qf_n=\sum_{t=1}^q\sum_{s=1}^{\min(t,q-1)}{\bf 1}(2q-t-s=p)\,(t-1)!{q-1\choose t-1}^2{t-1\choose s-1}f_n\ts_t^sf_n.$$ Indeed, use \eqref{eq:D2Proof}, the orthogonality of the random variables $$\int_{\sZ}I_{p_1}\big(G_{p_1}^{q-1}f_n(z,\,\cdot\,)\big)\,\mu_n(\dint z)\quad{\rm and}\quad\int_{\sZ}I_{p_2}\big(G_{p_2}^{q-1}f_n(z,\,\cdot\,)\big)\,\mu_n(\dint z)$$ for $1\leq p_1\neq p_2\leq 2(q-1)$ as well as the stochastic Fubini theorem \cite[Theorem 5.13.1]{PeccatiTaqquBook} (which is valid thanks to our technical assumptions made in Section \ref{sec:background}) to conclude that the identity \eqref{eq:D4Proof} is verified; see also the proof of Theorem 4.2 in \cite{PSTU2010}. We now exploit the assumption that $\|f_n\s_r^\ell f_n\|\to 0$ with $r$ and $\ell$ as in the statement of Theorem \ref{thm:GammaContractions}. It implies that
\begin{equation}\label{eq:OtherTermsGoToZero}
\langle f_n\ts_p^{q-p}f_n,f_n\rangle\to 0\quad{\rm and}\quad\langle f_n\ts_t^sf_n,f_n\ts_{t'}^{s'}f_n\rangle\to 0
\end{equation}
as $n\to\infty$ for all $p\in\{q/2+1,\ldots,q\}$ and $t,t'\in\{1,\ldots,q\}$, $s\in\{1,\ldots,\min(t,q-1)\}$, $s'\in\{1,\ldots,\min(t',q-1)\}$ and $t,s,t',s'$ not equal to $q/2$ at the same time. Indeed, $$|\langle f_n\ts_p^{q-p}f_n,f_n\rangle|\leq\|f_n\ts_p^{q-p}f_n\|\;\|f_n\|\leq\|f_n\s_p^{q-p}f_n\|\;\|f_n\|\to 0$$ for $p\in\{q/2+1,\ldots,q\}$ and similarly $$|\langle f_n\ts_t^sf_n,f_n\ts_{t'}^{s'}f_n\rangle|\leq\|f_n\ts_t^sf_n\|\;\|f_n\ts_{t'}^{s'}f_n\|\leq\|f_n\s_t^sf_n\|\;\|f_n\s_{t'}^{s'}f_n\|\to 0,$$ where $t,s,t',s'$ are as above.
Plugging the expressions for $\EE[\|DI_q(f_n)\|]$, $\EE\big[I_q(f_n)\|DI_q(f_n)\|^2\big]$ and $\EE[\|DI_q(f_n)\|^2]$ into \eqref{eq:ProofLemma2Condition} and using the first statement in \eqref{eq:OtherTermsGoToZero} we see immediately that \eqref{eq:ProofLemma2Condition} has the same limit as
\begin{equation}\label{eq:ProofLemma2ConditionII}
 8q^2\nu-2q^2\left({q\over 2}\right)!{q\choose q/2}^2q!\langle f_n\ts_{q/2}^{q/2}f_n,f_n\rangle+q^4\sum_{p=1}^{2(q-1)}p!\|\hat G_p^qf_n\|^2.
\end{equation}
We notice now that the middle contraction in the sum in \eqref{eq:ProofLemma2ConditionII} can only appear in the term $p=q$. Using the definition of $\hat G_q^qf_n$ and the second statement in \eqref{eq:OtherTermsGoToZero} we see that $q^4\,q!\|\hat G_q^qf_n\|^2$ behaves asymptotically like $$q^4\,q!\left(\left({q\over 2}-1\right)!\right)^2{q-1\choose q/2-1}^4\|f_n\ts_{q/2}^{q/2}f_n\|^2.$$ Consequently, \eqref{eq:ProofLemma2ConditionII} has the same limit as
\begin{equation*}
\begin{split}
& \big(8q^2\nu-4q^2\,q!\|f_n\|^2\big)+4q^2\,q!\|f_n\|^2-2q^2\left({q\over 2}\right)!{q\choose q/2}^2q!\langle f_n\ts_{q/2}^{q/2}f_n,f_n\rangle\\
&\hspace{4cm}+q^4\,q!\left(\left({q\over 2}-1\right)!\right)^2{q-1\choose q/2-1}^4\|f_n\ts_{q/2}^{q/2}f_n\|^2\to 0,
\end{split}
\end{equation*}
as $n\to\infty$, where we have used the fact that $\|f_n\ts_{q/2}^{q/2}f_n\|^2\to {2\over q!}c^2_q \nu$, and $\langle f_n\ts_{q/2}^{q/2}f_n,f_n\rangle\to {2\over q!}c_q\nu$. This proves the claim.
\end{proof}

We eventually deal with the convergence of the sequence
$$
A_3(I_q(f_n)) =\frac{1}{q} \int_{\sZ}\EE\big[(D_z{\bf 1}_{\{I_q(f_n)>-\nu\}})D_zI_q(f_n)|D_zI_q(f_n)|\big]\,\mu_n(\dint z),\qquad n\geq 1.
$$
\begin{lemma}
Under the conditions of Theorem \ref{thm:GammaContractions} we have that $ A_3(I_q(f_n)) \to 0$, as $n\to\infty$.
\end{lemma}
\begin{proof} In view of the assumptions, we can directly apply Proposition \ref{l:newbound}. 
It follows that our claim is proved once we show that the three terms on the right-hand side of (\ref{e:art}) (with $F=F_n = I_q(f_n)$ and $\mu = \mu_n$) converge to zero as $n\to\infty$. Since the first term equals $A_4(F_n)$, by virtue of the previous Lemma \ref{lem:GammaStep2}, we only have to prove the convergence of the remaining two summands. Our starting point is the following representation of the quantity $(D_{z_2}D_{z_1} F_n)^2 = q^2(q-1)^2 I_{q-2}^2(f_n(z_1,z_2,\,\cdot\,))$, which is obtained by means of the product formula (\ref{eq:ProductFormula}). Indeed,
\begin{equation*}
\begin{split}
(D_{z_2}D_{z_1} F_n)^2 &= q^2(q-1)^2\sum_{r=0}^{q-2}\sum_{\ell = 0}^{r} r!\binom{q-2}{r}^2\binom{r}{\ell} I_{2(q-2)-r-\ell} (f_n({z_1},{z_2}, \,\cdot\,)\star_r^{\ell}f_n({z_1},{z_2}, \,\cdot\,))\\
&=q^2(q-1)^2 I_{q-2}^2(f_n({z_1},{z_2},\,\cdot\,)).
\end{split}
\end{equation*}
Combining this representation with an iterated application of the triangle inequality, as well as of the isometric properties of multiple integrals, one deduces that the quantity $$\sqrt{\EE\int_{\mathcal Z}\int_{\mathcal Z} (D_{z_2}D_{z_1}F_n)^4\,\mu_n(\dint z_1)\mu_n(\dint z_2)}$$ is bounded by a linear combination (with coefficients not depending on $n$) of quantities of the type
$$
\sqrt{\int_{\mathcal Z}\int_{\mathcal Z} \| f_n(z_1,z_2, \cdot) \star_r^\ell f_n(z_1,z_2,\cdot)\|^2  \, \mu_n(\dint z_1)\mu_n(\dint z_2)} = \| f_n \star_{q-\ell}^{q-2-r} f_n\|\to 0,
$$
where the equality follows from a standard application of Fubini's theorem, and the convergence to zero is a consequence of the fact that $a:=q-\ell\in \{2,\ldots,q\}$ and $b :=q-2-r \in \{0,\ldots,a-2\}$, as well as of the elementary identity $\|f_n \star_a^0 f_n\| = \|f_n\star_q^{q-a} f_n\|$ ($2\leq a\leq q$). To deal with the remaining middle term, we use Fubini's theorem and the Cauchy-Schwarz inequality to deduce the estimate
$$
\sqrt{\EE\int_{\mathcal Z}\int_{\mathcal Z} (D_{z_2}D_{z_1}F_n)^2(D_{z_1}F_n)^2\,\mu_n(\dint z_1)\mu_n(\dint z_2)}\leq A_4(F_n)^{1/2} \times C^{1/4}_n,
$$
with $$C_n := \EE\int_{\mathcal Z}\left( \int_{\mathcal Z} (D_{z_2}D_{z_1}F_n)^2 \,\mu_n(\dint z_2)\right)^2 \mu_n(\dint z_1).$$
Using again the explicit representation of $(D_{z_2}D_{z_1}F_n)^2$ and applying several times Fubini's theorem, one sees that $C_n$ is indeed equal to a linear combination (with coefficients not depending on $n$) of objects of the type
$$
\|f_n\star_a^b f_n\|^2,\quad{\rm with}\quad a=2,\ldots,q\quad{\rm and}\quad b= 0,\ldots, a-2.
$$
The conclusion follows immediately since our estimates do not involve the middle contraction.
\end{proof}

\subsection{Proof of Proposition \ref{p:wi2}}

The product formula \eqref{eq:ProductFormula} shows that
\begin{eqnarray}\label{e:doubleprod}
I_2^2(f_n) &=& I_4({f_n\ts_0^0 f_n}) +4I_3({f_n\ts_1^0 f_n}) + I_2(4f_n\s_1^1 f_n +2 f^2_n) +4I_1(f_n\s_2^1f_n) +2\|f_n\|^2.
\end{eqnarray}
Using the relation
\begin{equation}\label{e:us}
4!\|{f_n\ts_0^0 f_n}\|^2 = 2(2\|f_n\|^2)^2 + 16\|f_n\s^1_1 f_n\|^2
\end{equation}
(see e.g.\ \cite[formula (5.2.12)]{NPBook}), exploiting the orthogonality of multiple integrals of distinct orders and using the fact that $\|f_n^2\|\to 0$ by assumption, we infer that $ \EE[I_2^4(f_n)] - 12\EE[I_2^3(f_n)] $ has the same limit as
\begin{eqnarray*}
&& 16\times 3!\|{f_n \ts_1^0 f_n}\|^2 +16\|f_n\s_2^1f_n\|^2 +48\| f_n\s_1^1 f_n\|^2  -96\langle f_n\s_1^1 f_n, f_n\rangle+   3(2\|f_n\|^2)^2 \\
&& =16\times 3!\|{f_n \ts_1^0 f_n}\|^2 +16\|f_n\star_2^1f_n\|^2+ 48\|f_n\s_1^1 f_n -f_n\|^2 -48\|f_n\|^2+3(2\|f_n\|^2)^2.
\end{eqnarray*}
The conclusion follows by observing that $\|f_n\|^2 \to \nu$ by assumption, and then by applying Theorem \ref{thm:GammaContractions}.\hfill $\Box$

\subsection{Proof of Theorem \ref{t:gdj}}

\noindent{\it Proof of  Part A}. According to Lemma \ref{prop:L1L2}, since each $\widetilde{F}'_n$ is completely degenerate, one has that $\widetilde{F}'_n = I_2(f_n)$, where $f_n = h_n/\sigma(n)$, and the double integral is performed with respect to the compensated Poisson measure $\hat{\eta}_n = \eta_n - \mu_n$. It follows that the estimate (\ref{e:e1}) is a direct consequence of \cite[Theorem 4.2]{PSTU2010}. Using formulae (\ref{e:doubleprod}) and (\ref{e:us}), we deduce that 
\begin{eqnarray*}
 \EE[I_2^4(f_n)] &=& 16\times 3!\|{f_n \ts_1^0 f_n}\|^2 +16\|f_n\star_2^1f_n\|^2 +16\| f_n\s_1^1 f_n\|^2 \\ &&\qquad +\,2\| 4f_n\s_1^1 f_n + 2f_n^2\|^2+   3(2\|f_n\|^2)^2,
\end{eqnarray*}
where the norms and contractions are of course taken with respect to the measure $\mu_n$. Since $3(2\|f_n\|^2)^2$ converges to $3$ by assumption, we deduce that, if (\ref{e:c1}) is verified, then the right-hand side of (\ref{e:e1}) converges to zero, and therefore $\widetilde{F}'_n$ converges in distribution to $N$. To conclude, observe that the estimates contained in \cite[pp.\ 744-745]{DyMa} yield that $\EE[(\widetilde{F}'_n - \widetilde{F}_n )^2] = O(n^{-1/2})$ as $n\to\infty$, so that the estimate (\ref{e:e2}) follows from the elementary inequality
$$
d_W( \widetilde{F}_n, N)\leq  d_W( \widetilde{F}'_n, N) + [\EE(\widetilde{F}'_n - \widetilde{F}_n )^2]^{1/2}.
$$

\bigskip

\noindent{\it Proof of Part B}. Again in view of Lemma \ref{prop:L1L2} and of the complete degeneracy of each $F'_n$, we deduce that $F'_n = I_2(h_n)$, where the double integral is again with respect to the compensated Poisson measure corresponding to $\eta_n$. The estimate (\ref{e:e3}) is therefore a consequence of Theorem \ref{thm:GammaContractions}, and the fact that the distribution of $\widetilde{F}'_n$ converges to $\cGam_\nu$ is a direct consequence of Proposition \ref{p:wi2} in the case $h_n = f_n$. The conclusion follows once again from the fact that $\EE[({F}'_n - {F}_n )^2] = O(n^{-1/2})$ as $n\to\infty$, in such a way that (\ref{e:e4}) follows from the triangle inequality
$$
d_3( \widetilde{F}_n, \cGam_\nu)\leq  d_3(\widetilde{F}'_n, \cGam_\nu) + [\EE(\widetilde{F}'_n -\widetilde {F}_n )^2]^{1/2}.
$$
This completes the proof.\hfill $\Box$

\subsection{Proof of Theorem \ref{t:gammaU}}

According to \cite[Theorem 7.3]{LRP2}, one has that 
$$\widetilde{F}'_n = I_2(h_n) + R_n,$$
where $h_n = h_2/n = n^{-1}\sum_{i=1}^\nu e_i\otimes e_i$, the double integral is realized with respect to the compensated Poisson measure $\hat{\eta}_n = \eta_n - n\mu$, and $R_n$  is a residual sequence of random variables such that
$$
\EE[R^2_n] = O(1/n),\qquad{\rm as}\qquad n\to\infty.
$$
It is immediate to verify that: (a) $\big(\int_{\sZ}h_n^4\,\dint\mu_n^2\big)^{1/4} = O(1/\sqrt{n})$ as $n\to\infty$, (b) $h_n \s_1^1 h_n = h_n$ (where the contraction is realized with respect to $\mu_n$), (c) $\|h_n \s_2^1 h_n\| = O(n^{-1/2})$ as $n\to\infty$ (since $h_n \s_2^1 h_n (x) = n^{-1}\sum_{i=1}^\nu e_i(x)^2$). The estimates are therefore a consequence of Theorem \ref{t:gdj}-(B), as well as of the estimates $\EE[({F}'_n - {F}_n )^2] = O(n^{-1/2})$ as $n\to\infty$ and 
$$
d_3( \widetilde{F}_n, \cGam_\nu)\leq  d_3(\widetilde{F}'_n, \cGam_\nu) + [\EE(\widetilde{F}'_n -\widetilde {F}_n )^2]^{1/2}.
$$
This completes the proof.\hfill $\Box$

\subsection{Proof of Theorem \ref{thm:MultivariateGamma}}

We start with some general preliminaries which will be specialized below.
Let $F_n^{(1)},\ldots,F_n^{(d)}$ be centered square-integrable functionals of the
Poisson measure $\eta$ in the domain of the derivative operator $D$. For $i\in\{1,\ldots,d\}$ let us define
\begin{equation}\label{eq:DefAlphaI}
\begin{split}
\alpha_n^{(i)} :=\EE|2(F_n^{(i)}+\nu_i)_+&-\langle DF_n^{(i)},-DL^{-1}F_n^{(i)}\rangle|+\EE\int_{\sZ}|D_zF_n^{(i)}|^2|D_zL^{-1}F_n^{(i)}|\,\mu_n(\dint z)\\
 &+\,\EE\int_{\sZ}(D_z{\bf 1}_{\{F_n^{(i)}>-\nu_i\}})(D_zF_n^{(i)})|D_zL^{-1}F_n^{(i)}|\,\mu_n(\dint z),
\end{split}
\end{equation}
and for $i\neq j\in\{1,\ldots,d\}$ put
\begin{eqnarray}\label{eq:DefBetaIJ}
\beta_n^{(i,j)}&:=&\EE| \langle  DF_n^{(i)} , DL^{-1}F_n^{(j)}\rangle |,\qquad\gamma_n^{(i,j)}:=\EE\int_{\sZ}|D_zF_n^{(i)}|^2|D_zL^{-1}F_n^{(j)}|\,\mu_n(\dint z).
\end{eqnarray}
We estimate the distance between (the law of) ${\bf
F}_n:=\big(F_n^{(1)},\ldots,F_n^{(d)}\big)$ and (that of) ${\bf
\Gamma}:=\big(G_1,\ldots,G_d\big)$ by $d({\bf F}_n,{\bf
\Gamma})=\sup|{\Bbb E}\phi({\bf F}_n)-{\Bbb E}\phi({\bf \Gamma})|$, where
the supremum runs over all functions $\phi:{\Bbb R}^d\to{\Bbb R}$ whose
partial derivatives up to order $3$ are bounded, continuous and satisfy $\|\,\cdot\,\|_\infty\leq 1$. We notice
that if $d({\bf F}_n,{\bf \Gamma})\to 0$ then ${\bf
F}_n\stackrel{d}{\to}{\bf \Gamma}$ as $n\to\infty$.

\begin{lemma}\label{lem:MultivariateHelp}
There exist constants $K_1$ and $K_2$ such that
\begin{equation*}
d({\bf F}_n,{\bf \Gamma})\leq
K_1\sum_{i=1}^d\alpha_n^{(i)}+K_2\sum_{\stackrel{i,j=1}{i\neq
j}}^d\big(\beta_n^{(i,j)}+\gamma_n^{(i,j)} \big).
\end{equation*}
\end{lemma}
\begin{proof}
The technique adopted here is similar to the one used in the proof of the main result of \cite{BourPec}. To keep the argument more transparent and the formulas simpler we restrict ourselves
to the case $d=2$, the general case can be dealt with similarly. So, ${\bf
F}_n=(F_n^{(1)},F_n^{(2)})$ and ${\bf \Gamma}=(G_1,G_2)$ and we have to
show that
\begin{equation}\label{eq:BoundMultivariateGeneral}
d\big((F_n^{(1)},F_n^{(2)}),(G_1,G_2)\big)\!\leq\!
K_1\big(\alpha_n^{(1)}+\alpha_n^{(2)}\big)\!+\!K_2\big(\beta_n^{(1,2)}\!+\!\beta_n^{(2,1)}\!+\!\gamma_n^{(1,2)}\!+\!\gamma_n^{(2,1)} \big).
\end{equation}
To accomplish this task, we shall provide uniform estimates on $|{\Bbb E}\phi(F_n^{(1)},F_n^{(2)})-{\Bbb
E}\phi(G_1,G_2)|$. First write
\begin{equation*}
\begin{split}
|{\Bbb E}\phi(F_n^{(1)},F_n^{(2)})-{\Bbb
E}\phi(G_1,G_2)|\leq|\EE[\phi(F_n^{(1)},F_n^{(2)})]&-\EE[\phi(G_1,F_n^{(2)})]|\\
&+|\EE[\phi(G_1,F_n^{(2)})]-\EE[\phi(G_1,G_2)]|=:|T_1|+|T_2|.
\end{split}
\end{equation*}
We first deal with $T_2$. Conditioning on $G_1$, we are in a one-dimensional situation and can
proceed as in the proof of Theorem \ref{thm:SteinBound}. This shows that
$T_2$ contributes the term $\alpha_n^{(2)}$ to the bound
\eqref{eq:BoundMultivariateGeneral}. We now consider the term $T_1$ and
write $\cL_U$ for the law of a random object $U$. Rewriting yields
$$T_1=\int\left(\phi(x,y)-\int\phi(g,y)\,\cL_{G_1}(\dint
g)\right)\cL_{(F_n^{(1)},F_n^{(2)})}\big(\dint(x,y)\big).$$ For fixed $y$ we
consider the term in brackets as the left-hand side of a Stein-equation
for the $\cGam_{\nu_1}$-distribution so that
\begin{equation}\label{eq:HelpSteinEquation}
\begin{split}
\int\left(\phi(x,y)-\int\phi(g,y)\,\cL_{G_1}(\dint
g)\right)&\cL_{(F_n^{(1)},F_n^{(2)})}\big(\dint(x,y)\big)\\ &=\int
2(x+\nu_1)_+h_y'(x)-xh_y(x)\,\cL_{(F_n^{(1)},F_n^{(2)})}\big(\dint(x,y)\big),
\end{split}
\end{equation}
where, for fixed $y$, $h_y(x)$ is the solution of the Stein-equation associated with the test function $x\mapsto \phi(x,y)$. We now consider the bivariate function $\hat h(x,y):=h_y(x)$. Using the smoothness assumptions on $\phi$ together with the explicit representation
\begin{equation*}\label{eq:ghat}
\hat h(x,y)=\begin{cases}-{1\over x}\big(\phi(x,y)-\EE[\phi(G(\nu_1),y)]\big) &:\,x\leq -\nu_1\\ {1\over 2(x+\nu_1)_+g_{\nu_1}(x)}\int_{-\nu_1}^x\big(\phi(z,y)-\EE[\phi(G(\nu_1),y)]\big)\,g_{\nu_1}(z)\,\dint z &: x>-\nu_1,\end{cases}
\end{equation*}
(recall the discussion preceding the proof of Theorem \ref{thm:SteinBound} and notice that $g_{\nu_1}(\,\cdot\,)$ stands for the density of the law $\cGam_{\nu_1}$) we deduce the following facts: (i) the mapping $x\mapsto \hat h(x,y)$ (for fixed $y$) is twice differentiable on $\RR\backslash\{-\nu\}$ (and it also admits right and left first derivatives at $x=-\nu$), and (ii) the mapping $y\mapsto \hat h(x,y)$ (for fixed $x$) is twice differentiable on $\RR$. All the involved derivatives are bounded by a finite constant only depending on $\nu_1$. Note that, in order to establish the estimates on $y\mapsto \hat h(x,y)$, one has to take derivatives under the integral and expectation signs, which is allowed thanks to the assumptions on $\phi$.

After these technical considerations we observe that
\eqref{eq:HelpSteinEquation} may be expressed in terms of $\hat h$ as
\begin{equation}\label{eq:HelpSteinEquationXXX}
\begin{split}
\int 2(x+\nu_1)_+\partial_1\hat h(x,y)&-x\hat
h(x,y)\,\cL_{(F_n^{(1)},F_n^{(2)})}\big(\dint(x,y)\big)\\
&=\EE\big[2(F_n^{(1)}+\nu_1)_+\partial_1\hat h(F_n^{(1)},F_n^{(2)})-F_1\hat h(F_n^{(1)},F_n^{(2)})\big]\\
&=\EE\big[2(F_n^{(1)}+\nu_1)_+\partial_1\hat g(F_n^{(1)},F_n^{(2)})-\langle D\hat
h(F_n^{(1)},F_n^{(2)}),-DL^{-1}F_n^{(1)}\rangle\big],
\end{split}
\end{equation}
where $\partial_1$ stands for the partial derivative with
respect to the first coordinate and where we have applied the
integration by parts formula \eqref{dualDdelta} of Malliavin calculus in exactly the same way as in the proof
of Theorem \ref{thm:SteinBound}. Using the notation
$F_{n,z}^{(i)}(\eta)=F_n^{(i)}(\eta+\delta_z)-F_n^{(i)}(\eta)$ for
$i\in\{1,2\}$ and $z\in\sZ$, we may write
\begin{equation*}
\begin{split}
D_z\hat h(F_1,F_2) &=\hat h(F_{n,z}^{(1)},F_{n,z}^{(2)})-\hat
h(F_{n}^{(1)},F_{n}^{(2)})\\
&=\big(\hat h(F_{n,z}^{(1)},F_{n,z}^{(2)})-\hat
h(F_{n,z}^{(1)},F_{n}^{(2)})\big)+\big(\hat
h(F_{n,z}^{(1)},F_{n}^{(2)})-\hat h(F_{n}^{(1)},F_{n}^{(2)})\big)=:S_1+S_2
\end{split}
\end{equation*}
Thanks to the properties of $\hat h$ described above, we find that $$S_1=\partial_2\hat
h(F_{n,z}^{(1)},F_n^{(2)})D_zF_n^{(2)}+R^{(1)}(D_zF_n^{(2)})\quad{\rm
and}\quad S_2=\partial_1\hat
h(F_n^{(1)},F_n^{(2)})D_zF_n^{(1)}+R^{(2)}(D_zF_n^{(1)}),$$ where
$R^{(1)}$ and $R^{(2)}$ are such that $$|R^{(1)}(D_zF_n^{(2)})|\leq K_1^{(1)}|D_zF_n^{(2)}|^2$$ and $$|R^{(2)}(D_zF_n^{(1)})|\leq K_1^{(2)}|D_zF_n^{(1)}|^2+K_2^{(2)}(D_z{\bf 1}_{\{F_n^{(1)}>-\nu_1\}})(D_zF_n^{(1)}),$$ where $\partial_{11}$ and $\partial_{22}$,
respectively, denote the second derivative with respect to the first and
second coordinate and where $K_1^{(1)},K_1^{(2)},K_2^{(2)}$ are finite constants. Combining this with \eqref{eq:HelpSteinEquationXXX} and
taking the supremum over all $\phi$, we obtain the contributions
$\alpha_n^{(1)}$, $\beta_n^{(2,1)}$ and $\gamma_n^{(2,1)}$ in
\eqref{eq:BoundMultivariateGeneral}. Inverting the role of $F_n^{(1)}$ and
$F_n^{(2)}$ in the previous discussion gives the bound
\eqref{eq:BoundMultivariateGeneral}, with constants $K_1$ and $K_2$ only depending on $(\nu_1,\nu_2)$.
\end{proof}
\begin{proof}[Proof of Theorem \ref{thm:MultivariateGamma}] Let us define the random vector ${\bf
I}_n:=\big(I_{q_1}(f_n^{(i)}),\ldots,I_{q_d}(f_n^{(d)})\big)$. We shall
prove that $d({\bf I}_n,{\bf \Gamma})\to 0$ as $n\to\infty$. Lemma
\ref{lem:MultivariateHelp} implies that for this it is sufficient to check
that $\alpha_n^{(i)}\to 0$, $\beta_n^{(i,j)}\to 0$ and that
$\gamma_n^{(i,j)}\to 0$ as $n\to\infty$ for any combination of
$i$ and $j$. Under the assumptions in the statement, writing $F_n^{(i)} = I_{q_i}(f_n^{(i)})$ one has the following three facts for every $i=1,\ldots,d$: (a) $\alpha_n^{(i)}\to 0$, as $n\to \infty$, (b) as $n\to\infty$,
$$
\EE\int_{\mathcal{Z}} (D_zF_n^{(i)} )^4 \,\mu_n(\dint z)\to 0,
$$
and (c) the sequence 
$$
\EE\int_{\mathcal{Z}} (D_zF_n^{(i)} )^2 \,\mu_n(\dint z) = 
q_i^2\,\EE\int_{\mathcal{Z}} (DL^{-1}F_n^{(i)} )^2 \,\mu_n(\dint z), \qquad n\geq 1,
$$
is bounded. An application of the Cauchy-Schwarz inequality yields therefore that $\gamma_n^{(i,j)}\to 0$ for any allowed choice of $i$ and $j$. To check the
fact that $\beta_n^{(i,j)}\to 0$, we apply once more the Cauchy-Schwarz
inequality to obtain $$\beta_n^{(i,j)}\leq
q_i^2\left(\EE\left(\int_{\sZ}I_{q_i-1}\big(f_n^{(i)}(z,\,\cdot\,)\big)\,I_{q_j-1}\big(f_n^{(j)}(z,\,\cdot\,)\big)\,\mu_n(\dint
z)\right)^2\right)^{1/2}.$$ We use now the general product formula \eqref{eq:ProductFormula} for multiple
integrals to express
$$I_{q_i-1}\big(f_n^{(i)}(z,\,\cdot\,)\big)\,I_{q_j-1}\big(f_n^{(j)}(z,\,\cdot\,)\big)$$ as a sum of
multiple integrals and the stochastic version of Fubini's theorem allowing
us to exchange deterministic with stochastic integration; see
\cite[Theorem 5.13.1]{PeccatiTaqquBook}. By assumption, $q_i<q_j$. Using the triangle inequality several times yields
\begin{equation*}
\begin{split}
\left( \EE\left(\int_{\sZ}I_{q_i-1}\big(f_n^{(i)}(z,\,\cdot\,)\big)\,
I_{q_j-1}\big(f_n^{(j)}(z,\,\cdot\,)\big)\,\mu_n(\dint z)\right)^2\right)^{1/2}
\leq \sum_{r=1}^{q_i}\sum_{\ell=1}^{r}K(r,\ell,q_i,q_j)^{1/2}\,\|
f_n^{(i)}\ts_{r}^{\ell} f_n^{(j)}\|,
\end{split}
\end{equation*}
with the constant $K(r,\ell,q_i,q_j)$ given by
$$K(r,\ell,q_i,q_j)=(r-1)!{q_i-1\choose
r-1}{q_j-1\choose r-1}{r-1\choose\ell-1}(q_i+q_j-r-\ell)!.$$ The proof is completed by observing that (see \cite[Lemma 2.9]{PecZheng})
$$\|
f_n^{(i)}\ts_{r}^{\ell} f_n^{(j)}\|\leq\|f_n^{(i)}\s_{r}^{\ell}
f_n^{(i)}\|\,\|f_n^{(j)}\s_{r}^{\ell} f_n^{(j)}\|\to 0$$
for all choices of $i,j$, because of the assumptions in the theorem and the fact that $2q_i\neq q_j$ for $i\neq j$.
\end{proof}

\subsection{Proof of Theorem \ref{thm:HybridConvergence}}

We start again with some preliminaries. Let
$F_n^{(1)},\ldots,F_n^{(d_1+d_2)}$ be square integrable functionals of the
Poisson measure $\eta$. For $i\in\{d_1+1,\ldots,d_1+d_2\}$ let us
define
\begin{equation*}
\d_n^{(i)}:=\EE|1-\langle
DF_n^{(i)},-DL^{-1}F_n^{(i)}\rangle|+\EE\int_{\sZ}|D_zF_n^{(i)}|^2|D_zL^{-1}F_n^{(i)}|\,\mu_n(\dint
z)
\end{equation*}
and for $i\in\{1,\ldots,d_1\}$ let $\alpha_n^{(i)}$ be as in
\eqref{eq:DefAlphaI} and for $i,j\in\{1,\ldots,d_1+d_2\}$ let
$\beta_n^{(i)}$ and $\gamma_n^{(i)}$ be as in \eqref{eq:DefBetaIJ}. We
will estimate the  distance between (the law of) ${\bf
F}_n:=\big(F_n^{(1)},\ldots,F_n^{(d_1+d_2)}\big)$ and (that of) the hybrid
vector ${\bf H}:=\big(G_1,\ldots,G_{d_1},N_{d_1+1},\ldots,N_{d_2}\big)$ by
the hybrid distance $d_h({\bf F}_n,{\bf H})=\sup|{\Bbb E}\phi({\bf
F}_n)-{\Bbb E}\phi({\bf H})|$, where the supremum runs over all functions
$\phi:{\Bbb R}^{d_1+d_2}\to{\Bbb R}$ whose partial derivatives up to order
$3$ are bounded, continuous and satisfy $\|\,\cdot\,\|_\infty\leq 1$.

\begin{lemma}\label{lem:HybridHelp}
There exist constants $K_1$, $K_2$ and $K_3$ such that $$d_h({\bf F}_n,{\bf
H})\leq
K_1\sum_{i=1}^{d_1}\alpha_n^{(i)}+K_2\sum_{i=d_1+1}^{d_1+d_2}\d_n^{(i)}+K_3\sum_{\stackrel{i,j=1}{i\neq
j}}^{d_1+d_2}\big(\beta_n^{(i,j)}+\gamma_n^{(i,j)}\big).$$
\end{lemma}
\begin{proof}
This follows along the same lines of argumentation as the proof of Lemma
\ref{lem:MultivariateHelp}. For this reason the details are omitted.
\end{proof}

\begin{proof}[Proof of Theorem \ref{thm:HybridConvergence}]
We first use Lemma \ref{lem:GammaStep1} to see that because of
\eqref{eq:GammaContractionsConditionHybrid}, $\alpha_n^{(i)}\to 0$ for any
$i\in\{1,\ldots,d_1\}$. Next, we apply \cite[Theorem 5.1]{PSTU2010} to
infer that under \eqref{eq:NormalContractionsConditionHybrid},
$\d_i^{(n)}\to 0$ as $n\to\infty$ for any
$i\in\{d_1,+1,\ldots,d_1+d_2\}$. The remaining discussion of
$\beta_n^{(i,j)}$ and $\gamma_n^{(i,j)}$ is very similar to the
multivariate pure Gamma case so that $\beta_n^{(i,j)}\to 0$ and $\gamma_n^{(i,j)}\to
0$ for all $i\neq j\in\{1,\ldots,d_1+d_2\}$. In view of Lemma
\ref{lem:HybridHelp}, this completes the proof.
\end{proof}

\subsection{Proof of Theorem \ref{thm:HybridPGConvergence}}

We consider a measurable bounded test function $\phi :  \RR\times\mathbb{Z}_+ \to \RR$ such that $\phi$ has uniformly bounded derivatives up to the order three in the first variable. By a slight variation of the arguments leading to the proof of \cite[Theorem 2.1]{BourPec} one has that there exists a universal constant $K>0$ (independent of $n$) such that
$$
\big|\EE[\phi(I_q(f_n), H_n)] - \EE[\phi(G, P)]\big|\leq K\,(A_n+B_n+C_n+D_n),
$$
where (similar to $\alpha_n$ etc.\ above)
\begin{eqnarray*}
&& A_n :=\EE\left|2(F_n+\nu)_+-\langle DF_n,-DL^{-1}F_n\rangle\right|+\int_{\sZ}\EE[|D_zF_n|^2|D_zL^{-1}F_n |]\,\mu_n(\dint z),\\
&& B_n :=\big| \EE[H_n]- \lambda\big|+ \EE\left| \lambda - \langle DH_n, -DL^{-1}H_n\rangle\right|+ \int_{\mathcal{Z}} \EE\big| D_zH_n (D_zH_n-1) D_zL^{-1}H_n\big|\,\mu_n(\dint z),\\
&& C_n := \EE[\langle |DH_n|, |DI_q(f_n)|\rangle]
\end{eqnarray*}
and 
$$
D_n := \EE\int_{\sZ}(D_z{\bf 1}_{\{I_q(f_n)>-\nu\}})(D_zI_q(f_n))|D_zL^{-1}H_n|\,\mu_n(\dint z)
$$

In view of Theorem \ref{thm:GammaContractions} (as well as of the estimates leading to its proof), the assumptions in the statement imply that $A_n+B_n +D_n\to 0$, and, moreover, that
$$
\EE\int_{\mathcal{Z}} (D_zI_q(f_n))^4\,\mu_n(\dint z) \to 0\qquad{\rm as}\qquad n\to\infty.
$$
The conclusion is obtained by observing that, by virtue of H\"older's inequality, and since $DH_n$ takes values in $\mathbb{Z}$,
\begin{eqnarray*}
C_n &\leq& \left( \EE\int_{\mathcal{Z}} (D_zI_q(f_n))^4\,\mu_n(\dint z)\right)^{1/4} \times \left(\EE \int_{\mathcal{Z}} (D_zH_n)^{4/3}\,\mu_n(\dint z)\right)^{3/4}\\
&\leq &\left(\EE \int_{\mathcal{Z}} (D_zI_q(f_n))^4\,\mu_n(\dint z)\right)^{1/4} \times \left(\EE \int_{\mathcal{Z}} (D_zH_n)^{2}\,\mu_n(\dint z)\right)^{3/4}\to 0,
\end{eqnarray*}
where we have implicitly used assumption (\ref{e:non}).\hfill $\Box$

\subsection*{Acknowledgement} We are grateful to Matthias Schulte and Yvik Swan for useful discussions.

\end{document}